\documentclass{amsart}

\usepackage{amssymb}
\usepackage{hyperref}

\usepackage{environ}

\theoremstyle{plain}
\newtheorem{thm}{Theorem}[section]
\newtheorem{theorem}[thm]{Theorem}
\newtheorem*{theorem*}{Theorem}
\newtheorem{cor}[thm]{Corollary}

\newtheorem{lemma}[thm]{Lemma}

\newtheorem{prop}[thm]{Proposition}

\newtheorem*{fact*}{Fact}
\newtheorem*{question*}{Question}

\newtheorem{main}{Theorem}

\theoremstyle{remark}
\newtheorem*{rem}{Remark}
\newtheorem*{rems}{Remarks}
\newtheorem*{note}{Note}

\theoremstyle{definition}
\newtheorem{definition}[thm]{Definition}
\newtheorem*{convention}{Convention}

\newtheorem{example}[thm]{Example}
\newtheorem*{standing hypothesis}{Standing Hypothesis}

\newcommand{\thmref}[1]{Theorem~\ref{#1}}
\newcommand{\corref}[1]{Corollary~\ref{#1}}
\newcommand{\lemref}[1]{Lemma~\ref{#1}}
\newcommand{\propref}[1]{Proposition~\ref{#1}}
\newcommand{\defref}[1]{Definition~\ref{#1}}

\newcommand{\itemrefstar}[1]{(\ref*{#1})}
\newcommand{\figref}[1]{Figure~\ref{#1}}
\newcommand{\secref}[1]{Section~\ref{#1}}

\newcommand{\defn}[1]{\emph{#1}}

\newcommand{\Z}{\mathbb{Z}}

\newcommand{\R}{\mathbb{R}}

\renewcommand{\H}{\mathbb{H}}

\newcommand{\mS}{\mathbb{S}}
\newcommand{\into}{\hookrightarrow}
\renewcommand{\setminus}{\smallsetminus}

\DeclareMathOperator{\supp}{supp}

\DeclareMathOperator{\diam}{diam}

\DeclareMathOperator{\Isom}{Isom}

\newcommand{\set}[1]{\left\{#1\right\}}
\newcommand{\setp}[2]{\left\{#1 \mid #2\right\}}
\newcommand{\family}[1]{\left(#1\right)}
\newcommand{\abs}[1]{\left| #1 \right|}

\newcounter{listcounter}

\usepackage{etoolbox}

\usepackage{tikz}

\usepackage{microtype}

\newcommand{\res}[1]{\vert_{#1}}
\newcommand{\mae}[1]{\mbox{#1-a.e.}}

\DeclareMathOperator{\Lk}{Lk}
\newcommand{\link}[2]{\Lk (#1)}
\DeclareMathOperator{\Spaceofgeodesics}{\Sigma}
\newcommand{\Sp}[2]{\Spaceofgeodesics_{#1}}

\newcommand{\cl}[1]{\overline{#1}}
\newcommand{\bd}{\partial_{\infty}}
\newcommand{\double}[1]{#1 \times #1}
\newcommand{\dbX}{\double{\bd X}}
\newcommand{\Lam}{\Lambda}
\newcommand{\dbL}{\double{\Lam}}

\newcommand{\Leb}{\lambda}

\newcommand{\GE}{\mathcal{G}^E}
\newcommand{\GER}{\GE \times \R}
\newcommand{\Reg}{\mathcal R}
\newcommand{\RE}{\Reg^E}
\newcommand{\RER}{\RE \times \R}
\newcommand{\Zerowidth}{\mathcal{Z}}
\newcommand{\ZE}{\Zerowidth^E}
\newcommand{\Specialset}{\mathcal S}

\newcommand{\SXL}{S_{\Lam} X}%
\newcommand{\GEL}{\GE_{\Lam}}
\newcommand{\GELR}{\GEL \times \R}
\newcommand{\RL}{\Reg_{\Lam}}
\newcommand{\REL}{\RE_{\Lam}}
\newcommand{\RELR}{\REL \times \R}
\newcommand{\ZL}{\Zerowidth_{\Lam}}
\newcommand{\ZEL}{\ZE_{\Lam}}

\newcommand{\lmod}{\backslash}

\newcommand{\modgp}[2]{#2 \lmod #1}%
\newcommand{\modG}[1]{\modgp{#1}{\Gamma}}

\DeclareMathOperator{\QE}{\mathcal Q_{\RE}}

\DeclareMathOperator{\QEL}{\mathcal Q_{\REL}}

\DeclareMathOperator{\Busemanninclusion}{\iota}

\DeclareMathOperator{\pr}{pr}%

\DeclareMathOperator{\cratio}{B}
\DeclareMathOperator{\emap}{E}

\DeclareMathOperator{\flip}{flip}
\DeclareMathOperator{\shadow}{\mathcal O}

\title[Flat strips, Bowen-Margulis measures, and mixing for CAT(0) spaces]{Flat strips, Bowen-Margulis measures, and mixing of the geodesic flow for rank one CAT(0) spaces}

\author{Russell Ricks}
\address{University of Michigan, Ann Arbor, Michigan, USA}
\email{rmricks@umich.edu}

\thanks{This material is based upon work supported by the National Science Foundation under Grant Number NSF 1045119.}
\thanks{The author would like to thank Brandon Seward, Andrew Zimmer, and Peter Scott for discussing this paper, Scott Schneider for helping track down a reference, and especially Ralf Spatzier for his support and many helpful discussions.}

\begin{document}

\begin{abstract}
Let $X$ be a proper, geodesically complete CAT($0$) space under a proper, non-elementary, isometric action by a group $\Gamma$ with a rank one element.  We construct a generalized Bowen-Margulis measure on the space of unit-speed parametrized geodesics of $X$ modulo the $\Gamma$-action.

Although the construction of Bowen-Margulis measures for rank one nonpositively curved manifolds and for CAT($-1$) spaces is well-known, the construction for CAT($0$) spaces hinges on establishing a new structural result of independent interest:  Almost no geodesic, under the Bowen-Margulis measure, bounds a flat strip of any positive width.  We also show that almost every point in $\bd X$, under the Patterson-Sullivan measure, is isolated in the Tits metric.  (For these results we assume the Bowen-Margulis measure is finite, as it is in the cocompact case).

Finally, we precisely characterize mixing when $X$ has full limit set:  A finite Bowen-Margulis measure is not mixing under the geodesic flow precisely when $X$ is a tree with all edge lengths in $c \Z$ for some $c > 0$.  This characterization is new, even in the setting of CAT($-1$) spaces.

More general (technical) versions of these results are also stated in the paper.
\end{abstract}

\maketitle

\section{Introduction}

In this paper, we study the ergodic theory of the geodesic flow on proper CAT($0$) spaces under proper, isometric group actions (see e.g.~ \cite{ballmann} or \cite{bridson} for the geometric background and motivation on CAT($0$) spaces).  The main goals of this paper are to construct a generalized Bowen-Margulis measure and to precisely characterize mixing of the geodesic flow for this measure in terms of the geometry of the space (\thmref{main trees}).  The construction of this measure is not an immediate generalization from the manifold setting, but involves establishing two structural results of independent interest:  First, almost every point in $\bd X$, under the Patterson-Sullivan measure, is isolated in the Tits metric (\thmref{main isolated}).  Second, almost no geodesic, under the Bowen-Margulis measure, bounds a flat strip of any positive width (\thmref{main lonely}).  All these results hold for a large class of CAT($0$) spaces, including the cocompact, geodesically complete case; the main assumption is the existence of a rank one axis.

The existence of a rank one axis in a CAT($0$) space---that is, an axis of translation that does not bound a flat half-plane---forces rather strong properties for the topological dynamics (see e.g.~ \lemref{ping-pong} and \propref{hamenstadt}).  These properties play an essential role in our construction and study of Bowen-Margulis measures.  Yet the existence of a rank one axis may well be generic for CAT($0$) spaces.  Indeed, Ballmann and Buyalo (\cite{bb}) conjecture that every geodesically complete CAT($0$) space under a proper, cocompact, isometric group action that does not admit a rank one axis must either split nontrivially as a product, or be a higher rank symmetric space or Euclidean building.  Moreover, this conjecture has been proven (and is called the Rank Rigidity Theorem) in a few important cases, notably for Hadamard manifolds by Ballmann, Brin, Burns, Eberlein, and Spatzier (see \cite{ballmann} and \cite{burns-spatzier}) and for CAT($0$) cube complexes by Caprace and Sageev (\cite{caprace-sageev}).

The Bowen-Margulis measure was first introduced for compact Riemannian manifolds of negative sectional curvature, where Margulis (\cite{margulis}) and Bowen (\cite{bowen71}) used different methods to construct the measure of maximal entropy for the geodesic flow, shown by Bowen (\cite{bowen73}) to be unique.  Sullivan (\cite{sullivan79} and \cite{sullivan84}) gave a third method to obtain this measure---originally in constant negative curvature, but extended by Kaimanovich (\cite{kaim90})
to variable negative curvature.
Our construction follows Sullivan's approach generally, and is especially inspired by work in two prior classes of spaces:  Knieper's (\cite{knieper97}), where $\modG{X}$ is a compact Riemannian manifold of nonpositive sectional curvature, and Roblin's (\cite{roblin}), where $X$ is a proper CAT($-1$) space.

The main obstacle in our construction is the presence of nontrivial flat strips.  In rank one nonpositively curved Riemannian manifolds, these form a closed, nowhere-dense subset of the unit tangent bundle, and each carries its own volume form.  However, in rank one CAT($0$) spaces, a priori it might happen that every geodesic bounds a nontrivial flat strip; moreover, the strips themselves do not carry a natural Borel measure.  Our solution is to construct the Bowen-Margulis measure in two stages, and to prove the necessary structural results between stages.

More precisely, let $X$ be a proper \textnormal{CAT($0$)} space under a proper, non-elementary, isometric action by a group $\Gamma$ with a rank one element.  Let $\mu_x$ be a Patterson-Sullivan measure on the boundary $\bd X$ of $X$, and let $\GE \subset \dbX$ be the set of endpoint pairs of geodesics in $X$.  There is a $\Gamma$-invariant Borel measure $\mu$ on $\GE$ (called a geodesic current) associated to $\mu_x$.  We first construct a Borel measure $m$ of full support on $\GER$ (by taking the product with Lebesgue measure on $\R$); $m$ descends to a Borel measure $m_\Gamma$ on the quotient $\modG{(\GER)}$.

We usually work with $m_\Gamma$ finite; this gives recurrence on a set of full measure by Poincar\'e.  Note that $m_\Gamma$ is always finite in the cocompact case, however.  We prove the following structural result.

\newcommand{\stateisolated}
{%
  Let $X$ be a proper \textnormal{CAT($0$)} space under a proper, non-elementary, isometric action by a group $\Gamma$ with a rank one element.  If $m_{\Gamma}$ is finite, then \mae{$\mu_x$} $\xi \in \bd X$ is isolated in the Tits metric.
}%

\begin{main}[\thmref{isolated almost everywhere}]
\label{main isolated}
\stateisolated
\end{main}

As a corollary, the equivalence classes of higher rank geodesics have zero measure under $m$.  We then show that, in fact, \mae{$m$} geodesic bounds no flat strip of any positive width.  Clearly we need to prevent obvious counterexamples such as those of the form $X = X' \times [0, 1]$, where $X'$ is a CAT($0$) space; the condition we require is that some geodesic in $X$ with both endpoints in $\Lam$, the limit set of $\Gamma$ in $\bd X$, has zero width.  Write $\ZL$ for the set of such geodesics.  Note that our condition holds automatically in the geodesically complete case (see \propref{strong density}).

\newcommand{\statelonely}
{%
  Let $X$ be a proper \textnormal{CAT($0$)} space under a proper, non-elementary, isometric action by a group $\Gamma$ with a rank one element.  Suppose $m_{\Gamma}$ is finite and $\ZL$ is nonempty.  Then the set $\ZE \subseteq \GE$ of endpoint pairs of zero-width geodesics has full $\mu$-measure.
}%

\begin{main}[\thmref{a.e.-geodesic-is-lonely}]
\label{main lonely}
\statelonely
\end{main}

Thus ``most'' geodesics do not bound a flat strip of any positive width, hence we may finally define Bowen-Margulis measures (also denoted $m$ and $m_\Gamma$) on the generalized unit tangent bundle $SX$ of $X$---called the space of geodesics of $X$ by Ballmann (\cite{ballmann})---and on $\modG{SX}$, by an near-canonical ``a.e.~ homeomorphism'' $\pi_x \colon SX \to \GER$.

The classical argument by Hopf (\cite{hopf71}) is readily adapted to prove ergodicity of the geodesic flow:

\newcommand{\stateergodicity}[1]
{%
  Let $X$ be a proper \textnormal{CAT($0$)} space under a proper, non-elementary, isometric action by a group $\Gamma$ with a rank one element.  Suppose $m_{\Gamma}$ is finite and $\ZL$ is nonempty.  Then all the following hold:
  \begin{enumerate}
  \item
  \ifstrequal{#1}{true}
    {\label{erg: ergodicity}}
    {}
  The Bowen-Margulis measure $m_\Gamma$ is ergodic under the geodesic flow on $\modG{SX}$.
  \item
  \ifstrequal{#1}{true}
    {\label{erg: diagonal}}
    {}
  The diagonal action of $\Gamma$ on $(\GE, \mu)$ is ergodic.
  \item
  \ifstrequal{#1}{true}
    {\label{erg: p-s}}
    {}
  The Patterson-Sullivan measure $\mu_x$ is both unique---that is, $\family{\mu_x}_{x \in X}$ is the unique conformal density of dimension $\delta_{\Gamma}$, up to renormalization---and quasi-ergodic.
  \item
  \ifstrequal{#1}{true}
    {\label{erg: div. type}}
    {}
  The Poincar\'e series $P(s, p, q)$ diverges at %
$s = \delta_\Gamma$ (i.e., $\Gamma$ is of divergence type).
  \end{enumerate}
}%

\begin{main}[\thmref{ergodicity}]
\label{main ergodicity}
\stateergodicity{false}
\end{main}

Then we characterize mixing.  This is more subtle than ergodicity, as some examples are known to be mixing, and others are known to be non-mixing.  For $\modG{X}$ a compact Riemannian manifold, Babillot (\cite{babillot}) showed that $m_\Gamma$ is mixing on $\modG{SX}$.  However, it is easy to see that if $X$ is a tree with only integer edge lengths, then $m_\Gamma$ is not mixing under the geodesic flow.  Nevertheless, we prove that if $X$ is geodesically complete and has full limit set, but $m_\Gamma$ is finite and non-mixing, then $X$ is homothetic to such a tree:

\newcommand{\statetrees}[1]
{%
	Let $X$ be a proper, geodesically complete \textnormal{CAT($0$)} space under a proper, non-elementary, isometric action by a group $\Gamma$ with a rank one element.  Suppose $m_{\Gamma}$ is finite and $\Lam = \bd X$.  The following are equivalent:
		\begin{enumerate}
	\item
		\ifstrequal{#1}{true}
	{\label{condition: not mixing}}
	{}
	The Bowen-Margulis measure $m_\Gamma$ is not mixing under the geodesic flow on $\modG{SX}$.
		\item
		\ifstrequal{#1}{true}
	{\label{condition: length spectrum arithmetic}}
	{}
	The length spectrum is arithmetic---that is, the set of all translation lengths of hyperbolic isometries in $\Gamma$ must lie in some discrete subgroup $c\Z$ of $\R$.
		\item
		\ifstrequal{#1}{true}
	{\label{condition: cross-ratios arithmetic}}
	{}
	There is some $c \in \R$ such that every cross-ratio of $\QE$ lies in $c\Z$.
		\item
		\ifstrequal{#1}{true}
	{\label{condition: special tree}}
	{}
	There is some $c > 0$ such that $X$ is isometric to a tree with all edge lengths in $c \Z$.
		\end{enumerate}
}%

\begin{main}[\thmref{trees}]
\label{main trees}
\statetrees{false}
\end{main}

A few remarks are in order.  First, if $m_\Gamma$ is not mixing, it also fails to be weak mixing or topological mixing because $\modG{SX}$ factors continuously over a circle for the trees in \thmref{main trees}; thus weak mixing and topological mixing are equivalent to mixing in this setting (by \propref{strong density}, $m_\Gamma$ has full support under the hypotheses of \thmref{main trees}).  Second, geodesic completeness is necessary for the above geometric characterization of mixing:  Take four congruent hyperbolic right isosceles triangles with rationally related edge lengths, and glue them together to form a regular hyperbolic quadrilateral; gluing together opposing vertices gives a locally CAT($-1$) space with arithmetic length spectrum.

The equivalence between mixing and non-arithmeticity of the length spectrum is not new for CAT($-1$) spaces (see \cite{babillot} and \cite{roblin}).  However, the only known geodesically complete examples with arithmetic length spectrum are trees.  Babillot and Roblin raised the question of what CAT($-1$) spaces other than trees could be non-mixing under a proper, non-elementary action.  For compact rank one nonpositively curved manifolds, Babillot (\cite{babillot}) showed that the Bowen-Margulis measure is always mixing; this also holds for proper CAT($-1$) spaces when $\Gamma$ has a parabolic element (see \cite{roblin}) and for rank one symmetric spaces (see \cite{kim01}).

However, it can still be difficult to determine if the length spectrum is arithmetic.  For instance, it remains an open question whether the length spectrum is always non-arithmetic for non-compact manifolds with non-elementary fundamental group, even with curvature $\le -1$ (see \cite{roblin}).  \thmref{main trees} shows, in particular, that trees are in fact the only examples of arithmetic length spectrum among cocompact and geodesically complete CAT($0$) spaces.

We highlight one particular difficulty in characterizing mixing here, as it is not obvious.  For proper CAT($-1$) spaces, if the length spectrum is arithmetic, the limit set is totally disconnected (see \cite{roblin}).  One might suppose, for $X$ a proper, cocompact, geodesically complete CAT($-1$) space, that having totally disconnected boundary might make it a tree, hence by the characterization of topological mixing for trees one might prove something like \thmref{main trees}.  However, this line of argument does not work.  Ontaneda (see the proof of Proposition 1 in \cite{ontaneda}) described proper, geodesically complete CAT($0$) spaces that admit a proper, cocompact, isometric action of a free group---hence they are quasi-isometric to trees and, in particular, have totally disconnected boundary---yet are not isometric to trees.  Ontaneda's examples are Euclidean $2$-complexes, but one can easily adapt the construction to hyperbolic $2$-complexes instead.  Thus there are proper, cocompact, geodesically complete CAT($-1$) spaces with totally disconnected boundary that are not isometric to trees.

Now, one the main applications we have in mind is to the cocompact and geodesically complete case.  In this setting, all the hypotheses of Theorems \ref{main isolated}--\ref{main trees} hold.

\newcommand{\statecocompactness}
{%
  Let $X$ be a proper, geodesically complete \textnormal{CAT($0$)} space under a proper, cocompact, isometric action by a group $\Gamma$ with a rank one element, and suppose $X$ is not isometric to the real line. Then $X$ satisfies the hypotheses of Theorems \ref{main isolated}--\ref{main trees}:  $m_{\Gamma}$ is finite, the $\Gamma$-action is non-elementary, $\Lam = \bd X$, and $\ZL$ is nonempty.
}%

\begin{main}[\thmref{cocompactness results}]
\label{main cocompactness}
\statecocompactness
\end{main}

Finally, we observe that Theorems \ref{main isolated}--\ref{main trees} hold (in slightly altered form) more generally than just for the Bowen-Margulis measure.  Specifically, our methods naturally extend to any finite quasi-product measure $m_{\Gamma}$ on $\modG{SX}$.  In the final section of the paper, we state the results in this generality.

\section{Patterson's Construction}

First we recall the construction of Patterson-Sullivan measures.  This construction is standard and due to Patterson (\cite{patterson}) for Fuchsian groups.

\begin{standing hypothesis}
In this section, let $X$ be a proper metric space---that is, a metric space in which all closed metric balls are compact.  Let $\Gamma$ be an infinite group of isometries acting properly discontinuously on $X$---that is, for every compact set $K \subseteq X$, there are only finitely many $\gamma \in \Gamma$ such that $K \cap \gamma K$ is nonempty.
\end{standing hypothesis}

For $p,q \in X$, $s \in \R$, the Dirichlet series
\[P(s, p, q) = \sum_{\gamma \in \Gamma} e^{-s d(p, \gamma q)}\]
is called the Poincar\'e series associated to $\Gamma$.  The following observation is standard.

\begin{lemma}
The critical exponent
\[\delta_\Gamma = \inf \setp{s \ge 0}{P(s, p, q) < \infty}\]
does not depend on choice of $p$ or $q$.
\end{lemma}

We will work only in the case that $\delta_\Gamma$ is finite.  This assumption is quite mild, however, because if $\Gamma$ is finitely generated, then $\delta_\Gamma$ is finite.  In particular, if $X$ is connected, and $\Gamma$ acts cocompactly on $X$, then $\delta_\Gamma$ is finite.

\begin{definition}
Let $X$ be a proper metric space.  Write $C(X)$ for the space of continuous maps $X \to \R$, equipped with the compact-open topology.  Fix $p \in X$, and let $\Busemanninclusion_p \colon X \to C(X)$ be the embedding given by $x \mapsto [d(\cdot, x) - d(p, x)]$.  The \defn{Busemann compactification of $X$}, denoted $\bar X$, is the closure of the image of $\Busemanninclusion_p$ in $C(X)$.  %
\end{definition}

For $\xi \in \bar X$, technically $\xi$ is a function $\xi \colon X \to \R$.  However, one usually prefers to think of $\xi$ as a point in $X$ (if $\xi$ lies in the image of $\Busemanninclusion_p$) or in the \defn{Busemann boundary}, $\bd X = \bar X \setminus X$, of $X$.  Instead of working with the function $\xi \colon X \to \R$, we will work with the Busemann function $b_\xi \colon X \times X \to \R$ given by $b_{\xi} (x, y) = \xi (x) - \xi (y)$.  Note that $b_\xi$ (unlike $\xi \colon X \to \R$) does not depend on choice of $p \in X$.

The Busemann functions $b_{\xi}$ are $1$-Lipschitz in both variables and satisfy the \defn{cocycle property} $b_{\xi} (x, y) + b_{\xi} (y, z) = b_{\xi} (x, z)$.  Furthermore, $b_{\gamma \xi} (\gamma x, \gamma y) = b_{\xi} (x, y)$ for all $\gamma \in \Isom X$.

For a measure $\mu$ on $X$ and a measurable map $\gamma \colon X \to X$, we write $\gamma_* \mu$ for the pushforward measure given by $(\gamma_* \mu)(A) = \mu (\gamma^{-1}(A))$ for all measurable $A \subseteq X$.

\begin{definition}\label{conformal density}
A family $\family{\mu_p}_{p \in X}$ of finite Borel measures on $\bd X$ is called a \defn{conformal density of dimension $\delta$} if
\begin{enumerate}
\item \label{equivariance}
$\gamma_* \mu_p = \mu_{\gamma p}$ for all $\gamma \in \Gamma$ and $p \in X$, and
\item \label{Radon-Nikodym}
for all $p, q \in X$, the measures $\mu_p$ and $\mu_q$ are equivalent with Radon-Nikodym derivative
\[\frac{d\mu_q}{d\mu_p}(\xi) = e^{-\delta b_\xi (q, p)}.\]
\end{enumerate}
\end{definition}

The \defn{limit set} $\Lam$ of $\Gamma$ is defined to be the subset of $\bd X$ given by
\[\Lam = \setp{\xi \in \bd X}{\gamma_i x \to \xi \text{ for some } (\gamma_i) \subset \Gamma \text{ and } x \in X}.\]
For a Borel measure $\nu$ on a topological space $Z$, its \defn{support} is the set
\[\supp(\nu) = \setp{z \in Z}{\nu(U) > 0 \text{ for every neighborhood } U \text{ of } z \in Z}.\]
We say $\nu$ has \defn{full support} if $\supp(\nu) = Z$.  Note $\supp (\mu_p) = \supp (\mu_q)$ for all $p, q \in X$, for any conformal density $\family{\mu_p}_{p \in X}$.  Thus the support of $\family{\mu_p}_{p \in X}$ is well-defined.

The classical construction of Patterson extends to the following setting:

\begin{theorem}[Patterson]\label{patterson}
Let $\Gamma$ be an infinite group of isometries acting properly discontinuously on a proper metric space $X$, and suppose $\delta_\Gamma < \infty$.  Then the Busemann boundary of $X$ admits a conformal density of dimension $\delta_\Gamma$ with support in $\Lam$.
\end{theorem}

If $\family{\mu_p}_{p \in X}$ is a conformal density obtained by Patterson's construction, $\mu_x$ is called a \defn{Patterson-Sullivan measure} on $\bd X$.

\begin{convention}
Throughout this paper, $\mu_x$ will always refer to a Patterson-Sullivan measure on $\bd X$.
\end{convention}

It would be useful to know if $\supp(\mu_p) = \bd X$ (for some, equivalently every, $p \in X$).  If $X$ is a proper rank one CAT($0$) space and $\Gamma$ acts cocompactly, this turns out (\propref{rank one summary 3}) to be equivalent to the existence of a rank one axis in $X$.

\section{Rank of Geodesics in CAT(0) Spaces}
\label{cat0}

We recall some properties of rank one geodesics in CAT($0$) spaces.  We assume some familiarity with CAT($0$) spaces (\cite{ballmann} and \cite{bridson} are good references).  The results in this section are found in the existing literature and generally stated here without proof.  \propref{rank one summary} is not in the literature as stated, but will not surprise the experts.

\begin{standing hypothesis}
In this section, let $\Gamma$ be a group acting by isometries on a proper CAT($0$) space $X$.
\end{standing hypothesis}

For CAT($0$) spaces, the Busemann boundary is canonically homeomorphic to the visual boundary obtained by taking equivalence classes of asymptotic geodesic rays (see \cite{ballmann} or \cite{bridson}).  Thus we will write $\bd X$ for the boundary with this topology, and use either description as convenient.

A \defn{geodesic} in $X$ is an isometric embedding $v \colon \R \to X$.  An isometric embedding of an interval $I \subset \R$ into $X$ is called a \defn{geodesic segment}, and an isometric embedding of $[0, \infty)$ into $X$ a \defn{geodesic ray}.  The space $X$ is \defn{geodesically complete} (or, $X$ has the \defn{geodesic extension property}) if every geodesic segment in $X$ extends to a full geodesic in $X$.

A subspace $F$ of $X$ is called a \defn{flat} if $F$ is isometric to some Euclidean $n$-space $\R^n$.  A subspace $Y \subset X$ isometric to $\R \times [0, \infty)$ is called a \defn{flat half-plane}; note that half-planes are automatically convex.  Call a geodesic in $X$ \defn{rank one} if its image does not bound a flat half-plane in $X$.  If a rank one geodesic $v$ is an axis of an isometry $\gamma \in \Gamma$, we call $v$ a \defn{rank one axis} and $\gamma$ a \defn{rank one isometry}.

Angles are defined as follows:  Let $x \in X$.  For $y, z \in X \setminus \set{x}$, the \defn{comparison angle} $\overline{\angle}_x (y, z)$ at $x$ between $y$ and $z$ is the angle at the corresponding point $\overline{x}$ in the Euclidean comparison triangle $\overline{\triangle}$ for the geodesic triangle $\triangle$ in $X$.  If $v$ and $w$ are geodesic segments in $X$ with $v(0) = w(0) = x$, the \defn{angle} at $x$ between $v$ and $w$ is $\angle_x (v, w) = \lim_{s,t \to 0^+} \overline{\angle}_x (v(s), w(t))$.  For $p, q \in \cl X \setminus \set{x}$, the \defn{angle} at $x$ between $p$ and $q$ is $\angle_x (p, q) = \angle_x (v, w)$, where $v$ and $w$ are geodesic segments with $v(0) = w(0) = x$, $v(d(x,p)) = p$, and $v(d(x,q)) = q$.

For $\xi, \eta \in \bd X$, let $\angle (\xi, \eta) = \sup_{x \in X} \angle_{x} (\xi, \eta)$.  Then $\angle$ defines a complete CAT($1$) metric on $\bd X$; this metric induces a topology on $\bd X$ that is finer (usually strictly finer) than the standard topology.  The \defn{Tits metric}, $d_T$, on $\bd X$ is the path metric induced by $\angle$ (which may take the value $+\infty$).  The \defn{Tits boundary} of $X$ is $\bd X$, equipped with the Tits metric $d_T$.

Since $d_T \ge \angle$ by definition, we have that if $\xi, \eta$ are the endpoints of a geodesic, then $d_T (\xi, \eta) \ge \pi$.  The next two lemmas follow from Theorem II.4.11 in \cite{ballmann}.

\begin{lemma}\label{rank one > pi}
A pair of points $\xi, \eta \in \bd X$ is joined by a rank one geodesic in $X$ if and only if $d_T (\xi, \eta) > \pi$.
\end{lemma}

\begin{lemma}\label{semicontinuity of Tits metric}
The Tits metric is lower semicontinuous---that is, the Tits metric $d_T \colon \dbX \to [0, \infty]$ is lower semicontinuous with respect to the visual topology on $\bd X$.
\end{lemma}

A subspace $Y \subset X$ isometric to $\R \times [0, R]$ is called a \defn{flat strip of width $R$}.  The next lemma is fundamental to understanding rank one geodesics in CAT($0$) spaces.  It implies, in particular, that the endpoint pairs of rank one geodesics form an open set in $\dbX$.

\begin{lemma}[Lemma III.3.1 in \cite{ballmann}]\label{blem}
Let $w \colon \R \to X$ be a geodesic which does not bound a flat strip of width $R > 0$.  Then there are neighborhoods $U$ and $V$ in $\bar X$ of the endpoints of $w$ such that for any $\xi \in U$ and $\eta \in V$, there is a geodesic joining $\xi$ to $\eta$.  For any such geodesic $v$, we have $d(v, w(0)) < R$; in particular, $v$ does not bound a flat strip of width $2R$.
\end{lemma}

Now we turn to Chen and Eberlein's duality condition from \cite{chen-eberlein}.  It is based on $\Gamma$-duality of pairs of points in $\bd X$, introduced by Eberlein in \cite{eb72}.

For any geodesic $v \colon \R \to X$, denote $v^+ = \lim_{t \to +\infty} v(t)$ and $v^- = \lim_{t \to -\infty} v(t)$.

\begin{definition}
Two points $\xi, \eta \in \bd X$ are called \defn{$\Gamma$-dual} if there exists a sequence $(\gamma_n)$ in $\Gamma$ such that $\gamma_n x \to \xi$ and $\gamma_n^{-1} x \to \eta$ for some (hence any) $x \in X$.  We say \defn{Chen and Eberlein's duality condition holds} on $\bd X$ if $v^+$ and $v^-$ are $\Gamma$-dual for every geodesic $v \colon R \to X$.
\end{definition}

\begin{lemma}[Lemma III.3.3 in \cite{ballmann}]\label{ping-pong}
Let $\gamma$ be an isometry of $X$, and suppose $w \colon \R \to X$ is an axis for $\gamma$, where $w$ is a geodesic which does not bound a flat half-plane.  Then
\begin{enumerate}
\item
For any neighborhood $U$ of $w^-$ and any neighborhood $V$ of $w^+$ in $\overline X$ there exists $n > 0$ such that
\[\gamma^{k} (\overline X \setminus U) \subset V
\quad
\text{and}
\quad
\gamma^{-k} (\overline X \setminus V) \subset U
\quad
\text{for all }
k \ge n.\]
\item
For any $\xi \in \bd X \setminus \set{w^+}$, there is a geodesic $w_{\xi}$ from $\xi$ to $w^+$, and any such geodesic is rank one.  Moreover, for $K \subset \bd X \setminus \set{w^+}$ compact, the set of these geodesics is compact (modulo parametrization).
\end{enumerate}
\end{lemma}

The next proposition summarizes the situation for rank one CAT($0$) spaces (cf.~ \propref{rank one summary 2} and \propref{rank one summary 3}).

\begin{prop}\label{rank one summary}
Let $\Gamma$ be a group acting properly discontinuously, cocompactly, and isometrically on a proper, geodesically complete \textnormal{CAT($0$)} space $X$.  Suppose $X$ contains a rank one geodesic, and that $\abs{\bd X} > 2$.  The following are equivalent:
\begin{enumerate}
\item \label{rank1axis}  $X$ has a rank one axis.
\item \label{total duality}  Every pair $\xi, \eta \in \bd X$ is $\Gamma$-dual.
\item \label{duality}
Chen and Eberlein's duality condition holds on $\bd X$.
\item \label{minimality}  $\Gamma$ acts minimally on $\bd X$ (that is, every $p \in \bd X$ has dense $\Gamma$-orbit).
\item \label{infinite-diameter}  Some $\xi \in \bd X$ has infinite Tits distance to every other $\eta \in \bd X$.
\item \label{3pi/2-diameter}  $\bd X$ has Tits diameter $\ge \frac{3\pi}{2}$.
\end{enumerate}
\end{prop}
\begin{proof}
\itemrefstar{duality}%
$\implies$%
\itemrefstar{minimality} and
\itemrefstar{duality}%
$\implies$%
\itemrefstar{rank1axis} are shown in Ballmann (Theorems III.2.4 and III.3.4, respectively, of \cite{ballmann}).
\itemrefstar{rank1axis}%
$\implies$%
\itemrefstar{infinite-diameter} is an easy exercise using \lemref{ping-pong}(2), while
\itemrefstar{infinite-diameter}%
$\implies$%
\itemrefstar{3pi/2-diameter}
and
\itemrefstar{total duality}%
$\implies$%
\itemrefstar{duality}
are trivial.
\itemrefstar{3pi/2-diameter}%
$\implies$%
\itemrefstar{rank1axis} is shown (with slightly better bounds for any fixed dimension) in Guralnik and Swenson (\cite{sg}).
\itemrefstar{minimality}%
$\implies$%
\itemrefstar{total duality} follows immediately from Corollary 1.6 of Ballmann and Buyalo (\cite{bb}).

It remains to prove
\itemrefstar{rank1axis}%
$\implies$%
\itemrefstar{minimality}.  Let $p, q$ be the endpoints of a rank one axis, and let $M$ be a minimal nonempty closed $\Gamma$-invariant subset of $\bd X$.  By \lemref{ping-pong}(1), both $p, q$ must lie in $M$; thus $M$ is the only minimal set.  By Corollary 2.1 of Ballmann and Buyalo (\cite{bb}), the orbit of $p$ is dense in the boundary.  Since $p \in M$, this means the $\Gamma$-action is minimal on the boundary.
\end{proof}

A well-known conjecture of Ballmann and Buyalo (\cite{bb}) is that, given the hypotheses of \propref{rank one summary}, all the equivalent conditions in the conclusion hold.

\section{Patterson-Sullivan Measures on CAT($0$) Boundaries}

We make a few observations about Patterson-Sullivan measures for CAT($0$) spaces.

\begin{standing hypothesis}
In this section, let $\Gamma$ be an infinite group acting properly by isometries on a proper CAT($0$) space $X$.
\end{standing hypothesis}

\begin{definition}
Define the \defn{$r$-shadow of $y$ from $x$} to be
\[\shadow_r (x, y) = \setp{\xi \in \bd X}{[x, \xi) \cap B(y, r) \neq \varnothing},\]
where $[x, \xi)$ is the image of the geodesic ray from $x$ to $\xi$.
\end{definition}

The proof of the next lemma is the same as for CAT($-1$) spaces (see \cite{roblin}).  %

\begin{lemma}[Sullivan]\label{shadow lemma}
For every $r > 0$, there is some $C_r > 0$ such that
\[\mu_x (\shadow_r (x, \gamma x)) \le C_r e^{- \delta_\Gamma d(x, \gamma x)}\]
for all $x \in X$ and $\gamma \in \Gamma$.
\end{lemma}

\begin{prop}\label{flats}
If $\Gamma$ acts cocompactly on $X$ and $\delta_\Gamma > 0$, then $\mu_x (\bd F) = 0$ for any flat $F \subset X$.
\end{prop}
\begin{proof}
Let $F \subset X$ be a flat.  Fix $x \in X$; we may assume $x \in F$.  By cocompactness of the $\Gamma$-action, there is some $R > 0$ such that $\Gamma B(x, R) = X$.  Now the spheres
\[S_F (x, r) = \setp{y \in F}{d(y, x) = r}\]
in $F$ based at $x$ may be covered by at most $p(r)$ $R$-balls in $F$, for some polynomial function $p \colon \R \to \R$.  But the center of each of these balls lies within distance $R$ of some $\gamma x$ in $X$ ($\gamma \in \Gamma$).  Thus
\[S_F (x, r) \subset \bigcup_{\gamma \in A_r} B(\gamma x, 2R),\]
where $A_r \subset \Gamma$ has cardinality at most $p(r)$.  Now by \lemref{shadow lemma}, for every $r > 0$ we have
\begin{align*}
\mu_x (\bd F) &\le \sum_{\gamma \in A_r} \mu_x (\shadow_{2R} (x, \gamma x))
\le \sum_{\gamma \in A_r} C_{2R} \cdot e^{- \delta_\Gamma d(x, \gamma x)}
= C_{2R} \cdot e^{- \delta_\Gamma r} \abs{A_r}.
\end{align*}
Since $\abs{A_r} \le p(r)$, we therefore have $\mu_x (\bd F) \le C_{2R} \cdot e^{-\delta_\Gamma r} p(r)$ for all $r > 0$.  But $e^{-\delta_\Gamma r} p(r) \to 0$ as $r \to +\infty$ because $\delta_\Gamma > 0$ and $p(r)$ is polynomial.  Thus $\mu_x (\bd F) = 0$, as required.
\end{proof}

A point $\xi \in \Lam$ is called a \defn{radial} (or \defn{conical}) limit point of $\Gamma$ if there is a sequence $(\gamma_n) \subset \Gamma$ such that $\gamma_n x \to \xi$ with $\set{\gamma_n x}$ boundedly close to some geodesic ray for some (any) $x \in X$.  Observe that the proof of \propref{flats} also shows:

\begin{cor}\label{radial limits}
If $\delta_{\Gamma} > 0$ then $\mu_x (\set{\xi}) = 0$ for any radial limit point $\xi \in \Lam$.
\end{cor}

On the other hand, we have the following result.  (A group action by $\Gamma$ on $X$ is said to be \defn{elementary} if either $\abs{\Lam}$ contains at most two points, or $\Gamma$ fixes a point of $\bd X$.)

\begin{lemma}\label{adams-ballmann}
Suppose either (a) the $\Gamma$-action is non-elementary with a rank one element, or (b) $\Gamma$ acts cocompactly with $X$ geodesically complete.  If $\delta_\Gamma = 0$, then $X$ is flat---that is, $X$ is isometric to flat Euclidean $n$-space $\R^n$ for some $n$.
\end{lemma}
\begin{proof}
Suppose $\delta_\Gamma = 0$.  Then $\Gamma$ must have subexponential growth, so $\Gamma$ is amenable.  In particular, $\Gamma$ has no free subgroups of rank $\ge 2$, ruling out case (a).  By Adams and Ballmann (\cite[Corollary~C]{adams-ballmann}), $X$ is flat.
\end{proof}

\propref{flats} and \lemref{adams-ballmann} immediately give us the following corollary.

\begin{cor}\label{mu_x not atomic}
If $\Gamma$ acts cocompactly on $X$, with $X$ geodesically complete and not flat, then $\mu_x$ has no atoms---that is, $\mu_x (\set{\xi}) = 0$ for all $\xi \in \bd X$.
\end{cor}

\section{A Weak Product Structure}

We now study the space $SX$ of unit-speed parametrized geodesics in $X$.  Much of our work in later sections depends on a certain product structure on this space, which we will describe shortly.

\begin{standing hypothesis}
In this section, let $\Gamma$ be a group acting by isometries on a proper CAT($0$) space $X$.
\end{standing hypothesis}

Let $SX$ be the space of unit-speed parametrized geodesics in $X$, endowed with the compact-open topology, and let $\Reg \subseteq SX$ be the space of rank one geodesics in $SX$.  For $v \in SX$ denote $v^+ = \lim_{t \to +\infty} v(t)$ and $v^- = \lim_{t \to -\infty} v(t)$.  Let
\[\GE = \setp{(v^-, v^+) \in \dbX}{v \in SX}\]
and
\[\RE = \setp{(v^-, v^+) \in \dbX}{v \in \Reg}.\]
Note that $\RE$ is open in $\GE$ by \lemref{blem}, and the natural projection $\emap \colon SX \to \GE$ is a continuous surjection with $\Reg = \emap^{-1}(\RE)$, so $\Reg$ is open in $SX$.

Much of the dynamical information lies in the following sets:
\begin{gather*}
\SXL = E^{-1} (\dbL),
\qquad
\RL = \Reg \cap E^{-1} (\dbL), \\
\GEL = \GE \cap (\dbL),
\quad \text{and} \quad
\REL = \RE \cap (\dbL).
\end{gather*}

There are many metrics on $SX$ (compatible with the compact-open topology) on which the natural $\Gamma$-action $\gamma(v) = \gamma \circ v$ is by isometries.  For simplicity, we will use the metric on $SX$ given by
\[d(v, w) = \sup_{t \in \R} e^{-\abs{t}} d(v(t), w(t)).\]
Under this metric, the footpoint projection $\pi \colon SX \to X$ given by $\pi(v) = v(0)$ is a proper map and in fact $1$-Lipschitz.  Thus we obtain:

\begin{lemma}\label{SX is proper}
Under the metric given above, $SX$ is a proper metric space, and the $\Gamma$-action on $X$ naturally induces a $\Gamma$-action on $SX$ by isometries.  The induced action on $SX$ is proper if the action on $X$ is proper, and cocompact if the action on $X$ is cocompact.
\end{lemma}

For $p \in X$, define $\beta_p \colon \dbX \to [-\infty, \infty)$ by $\beta_p (\xi, \eta) = \inf_{x \in X} (b_{\xi} + b_{\eta}) (x, p)$.

\begin{lemma}\label{beta}
For any $\xi, \eta \in \bd X$, $\beta_p (\xi, \eta)$ is finite if and only if $(\xi, \eta) \in \GE$.  Moreover,
\[\beta_p (\xi, \eta) = (b_{\xi} + b_{\eta}) (x, p)\]
if and only if $x$ lies on the image of a geodesic $v \in \emap^{-1} (\xi, \eta)$.
\end{lemma}
\begin{proof}
This is shown in the proof of implications $(1) \implies (2)$ and $(2) \implies (1)$ of Proposition II.9.35 in \cite{bridson}.
\end{proof}

Thus we may (abusing notation slightly) also write $\beta_p \colon SX \to \R$ to mean the map $\beta_p \circ \emap$; that is, $\beta_p (v) = \beta_p (v^-, v^+) = (b_{v^-} + b_{v^+}) (v(0), p)$.

\begin{lemma}\label{beta is continuous}
For any $p \in X$, the map $\beta_p$ is continuous on $\RE$ and upper semicontinuous on $\dbX$.
\end{lemma}
\begin{proof}
Continuity on $\RE$ first.  Fix $p \in X$, and suppose $(v_n^-, v_n^+) \to (v_-, v^+)$.  By \lemref{blem}, we may assume that $d(v_n(0), v(0)) < R$ for some $R > 0$.  So by the Arzel\`a-Ascoli Theorem, we may pass to a subsequence such that $v_n \to u$ for some $u \in SX$.  Then $u$ must be parallel to $v$, hence $\beta_p (u) = \beta_p (v)$.  Define $c_w \colon X \to \R$ by $c_w (q) = (b_{w^-} + b_{w^+}) (q, p)$.  Thus $c_w (w(0)) = \beta_p (w)$ for all $w \in SX$.  Since $(v_n^-, v_n^+) \to (v_-, v^+)$, we have $c_{v_n} \to c_v$ uniformly on $\overline{B}(u(0), 1)$, and therefore $\set{c_{v_n}} \cup \set{c_v}$ is uniformly equicontinuous on $\overline{B}(u(0), 1)$.  Thus $v_n(0) \to u(0)$ gives us $c_{v_n} (v_n(0)) \to c_{v} (u(0))$.  But $c_{v} (u(0)) = c_{u} (u(0)) = \beta_p (u)$, and $c_{v_n} (v_n(0)) = \beta_p (v_n)$, hence $\beta_p (v_n) \to \beta_p (u)$.  Therefore, $\beta_p$ is continuous on $\RE$.

Now semicontinuity on $\dbX$.  Recall that $\beta_p (\xi, \eta) = \inf_{x \in X} (b_{\xi} + b_{\eta}) (x, p)$.  Fix $p \in X$, and note that for fixed $x \in X$, the map $(\xi, \eta) \mapsto (b_\xi + b_\eta)(x, p)$ is continuous.  But the infimum of a family of continuous functions is upper semicontinuous.
\end{proof}

For $v \in SX$, let $P_v$ be the set of $w \in SX$ parallel to $v$ (we will also write $w \parallel v$).  Let $X_v$ be the union of the images of $w \in P_v$.  Recall (\cite{ballmann} or \cite{bridson}) that $X_v$ splits as a canonical product $Y_v \times \R$, where $Y_v$ is a closed and convex subset of $X$ with $v(0) \in Y_v$.  Call $X_v$ the \defn{parallel core} of $v$ and $Y_v$ the \defn{transversal} of $X_v$ at $v$.

If $v$ is rank one, then $Y_v$ is bounded and therefore has a unique circumcenter (see \cite{ballmann} or \cite{bridson}).  Thus we have a canonical \defn{central geodesic} associated to each $X_v$.  Let $\Reg_C$ denote the subset of central geodesics in $\Reg$.

Now suppose $v_n \to v \in SX$.  Then $b_{v_n^-} \to b_{v^-}$ and $b_{v_n^+} \to b_{v^+}$ by coincidence of the visual and Busemann boundaries.  Furthermore, $b_{v_n^-} (v_n(0), x) \to b_{v^-} (v(0), x)$ for all $x \in X$ because $v_n(0) \to v(0)$ while $b_{v_n^-} \to b_{v^-}$ uniformly on $\overline{B} (v, 1)$.  This shows the map $\pi_x$ in the following definition is continuous.

\begin{definition}\label{weak-def}
Let $\pi_x \colon SX \to \GER \subseteq \dbX \times \R$ be the continuous map given by $\pi_x (v) = (v^-, v^+, b_{v^-} (v(0), x))$.  Say that a sequence $(v_n) \subset SX$ \defn{converges weakly} to $v \in SX$ if $\pi_x (v_n) \to \pi_x (v)$.
\end{definition}

For a sequence that converges in $SX$, we will sometimes say it \defn{converges strongly} to emphasize that the convergence is not in the weak sense.

\begin{note}
Weak convergence does not depend on choice of $x \in X$.
\end{note}

\begin{example}
Consider the hyperbolic plane $\H^2$.  Cut along a geodesic, and isometrically glue the two halves to the two sides of a flat strip of width $1$.  Call the resulting space $X$.  A sequence $v_n$ of geodesics in $X$ which converges strongly to one of the geodesics (call it $v$) bounding the flat strip will also converge weakly to all the geodesics $w$ parallel to $v$ such that $w(0)$ lies on the geodesic segment orthogonal to the image of $v$.  (See \figref{figure for weak-strong convergence example}.)
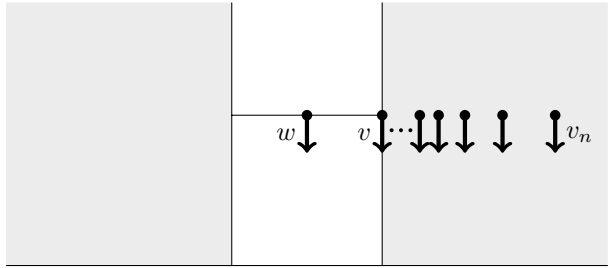
\begin{figure}[ht]
\begin{tikzpicture} %
  \fill [gray, opacity = 0.15]
    (0, 3.5) rectangle (3, 0)
    (8, 3.5) rectangle (5, 0);
  \fill [white]
    (3, 0) rectangle (5, 3.5); %
  \draw
    (0, 0) -- (8, 0)
    (3, 0) -- +(0, 3.5)
    (5, 0) -- +(0, 3.5)
    (3, 2) -- +(2, 0);
  \newcommand{\arrow}[1]
  {
    edge #1 +(0, -1/2)
    [fill] circle (2pt)
  }
  \path [->, ultra thick]
    (5, 2) \arrow{node [left] {$v$}}
    ++(-1, 0) \arrow{node [left] {$w$}}
    ++(1, 0)
    \foreach \x in {0.5, 0.25, 0.35, 0.5}
      {++(\x, 0) \arrow{}}
    ++(0.7, 0) \arrow{node [right] {$v_n$}};
  \fill
    (5, 2)
    \foreach \x in {1, 2, 3}
      {+(\x*0.5/4, -0.2) circle (0.75pt)};
\end{tikzpicture}
\caption{The geodesics $v_n$ converge weakly to both $v$ and $w$, but strongly to $v$ only.}
\label{figure for weak-strong convergence example}
\end{figure}
\end{example}

Let us now relate equivalence of geodesics in the product structure to the idea of stable and unstable horospheres, and to the transversals of parallel cores.

\begin{definition}
For $v \in SX$, the \defn{stable horosphere} at $v$ is the set of geodesics
\[H^s (v) = \setp{w \in SX}{w^+ = v^+ \text{ and } b_{v^+} (w(0), v(0)) = 0}.\]
Similarly, the \defn{unstable horosphere} is the set of geodesics
\[H^u (v) = \setp{w \in SX}{w^- = v^- \text{ and } b_{v^-} (w(0), v(0)) = 0}.\]
\end{definition}

\begin{prop}\label{equivalent-geodesics}
For $v, w \in SX$ and $x \in X$, the following are equivalent:
\begin{enumerate}
\item  $\pi_x (v) = \pi_x (w)$.
\item  $w \in H^u (v)$ and $v^+ = w^+$.
\item  $w \in H^s (v)$ and $v^- = w^-$.
\item  $w\in H^s (v) \cap H^u (v)$.
\item  $v \parallel w$ and $w(0) \in Y_v$.
\end{enumerate}
\end{prop}
\begin{proof}
We may assume throughout the proof that $v \parallel w$.  Since
\[(b_{v^-} + b_{v^+}) (v(0), x) = \beta_x (v) = \beta_x (w) = (b_{w^-} + b_{w^+}) (w(0), x),\]
we have $b_{v^-} (v(0), x) = b_{w^-} (w(0), x)$ if and only if $b_{v^+} (v(0), x) = b_{w^+} (w(0), x)$; this proves the equivalence of the first four conditions.  Recall (\cite[Proposition I.5.9]{ballmann}, or \cite[Theorem II.2.14(2)]{bridson}) that
$Y_v$ is preimage of $v(0)$ in $X_v$ under the orthogonal projection onto the image of $v$.  Now orthogonal projection onto the image of $v$ cannot increase either $b_{v^-} (\cdot, x)$ or $b_{v^+} (\cdot, x)$ by \cite[Lemma II.9.36]{bridson}, but $\beta_x (v) = \beta_x (w)$ because $v \parallel w$.  So for $w(t_0) \in Y_v$,
\[b_{v^-} (v(0), x) = b_{v^-} (w(t_0), x) = b_{w^-} (w(t_0), x) = t_0 + b_{w^-} (w(0), x).\]
Thus $\pi_x (v) = \pi_x (w)$ if and only if $w(0) \in Y_v$ (note $w(t_0) \in Y_v$ for only one $t_0 \in \R$).  This concludes the proof.
\end{proof}

We will write $u \sim v$ if $v$ and $w$ satisfy any of the equivalent conditions in the above proposition.  Clearly $\sim$ is an equivalence relation.  Note that by \propref{equivalent-geodesics}, this relation does not depend on choice of $x \in X$.

\begin{lemma}\label{bounded}
If $v_n \to v$ weakly and $v \in \Reg$, then $\set{v_n(0)}$ is bounded in $X$.
\end{lemma}
\begin{proof}
Fix $x \in X$.  Since $v_n^- \to v^-$, we have $b_{v_n^-} \to b_{v^-}$ uniformly on compact subsets, and so $b_{v_n^-} (v(0), x) \to b_{v^-} (v(0), x)$.  On the other hand, we know that $b_{v_n^-} (v_n(0), x) \to b_{v^-} (v(0), x)$ by hypothesis, so
\[\lim_{n \to \infty} b_{v_n^-} (v(0), x) = b_{v^-} (v(0), x) = \lim_{n \to \infty} b_{v_n^-} (v_n(0), x).\]
Hence, by the cocycle property of Busemann functions,
\[\lim_{n \to \infty} \left( b_{v_n^-} (v(0), v_n(0)) \right) = \lim_{n \to \infty} \left( b_{v_n^-} (v(0), x) - b_{v_n^-} (v_n(0), x) \right) = 0.\]
Now let $R > 0$ be large enough so that $v$ does not bound a flat strip in $X$ of width $R$.  By \lemref{blem}, for all sufficiently large $n$ there exist $t_n \in \R$ such that $d(v_n(t_n), v(0)) < R$.
Thus
\begin{align*}
\abs{t_n}
= \abs{b_{v_n^-} (v_n(t_n), v_n(0))}
&= \abs{b_{v_n^-} (v_n(t_n), v(0)) - b_{v_n^-} (v_n(0), v(0))} \\
&\le d(v_n(t_n), v(0)) + \abs{b_{v_n^-} (v_n(0), v(0))}
< R + 1
\end{align*}
for all sufficiently large $n$.  In particular,
\[d(v_n(0), v(0)) \le d(v_n(0), v_n(t_n)) + d(v_n(t_n), v(0)) < \abs{t_n} + R < 2R + 1.\qedhere\]
\end{proof}

\begin{lemma}\label{quotient}
If $v_n \to v$ weakly and $v \in \Reg$ then a subsequence converges strongly to some $u \sim v$.
\end{lemma}
\begin{proof}
Fix $x \in X$.  By \lemref{bounded}, $\set{v_n(0)}$ lies in some compact set in $X$.  Hence by the Arzel\`a-Ascoli Theorem, passing to a subsequence we may assume that $(v_n)$ converges in $SX$ to some geodesic $u$.  Then $\pi_x (u) = \lim \pi_x (v_n)$ by continuity of $\pi_x$, while $\pi_x (v) = \lim \pi_x (v_n)$ by hypothesis, and therefore $u \sim v$.
\end{proof}

\begin{rem}
Restricting $\pi_x$ to $\Reg_C$ does not automatically give us a homeomorphism from $\Reg_C$ to $\RER$. %
We get a topology on $\Reg_C$ at least as course as the subspace topology, though.  An explicit example of the failure of $\pi_x$ to be a homeomorphism is as follows:  Take a closed hyperbolic surface, and replace a simple closed geodesic with a flat cylinder of width $1$; then there are sequences of geodesics that limit, weakly but not strongly, onto one of the central geodesics in the flat cylinder.
\end{rem}

From \lemref{quotient}, we see that the continuous map $\pi_x \res{\Reg}$ is closed (that is, the image of every closed set is closed).  Thus $\pi_x \res{\Reg}$ is a topological quotient map onto $\RER$.

Let $g^t \colon SX \to SX$ denote the geodesic flow; that is, $(g^t (v))(s) = v (s + t)$.  Note that $g^t$ commutes with $\Gamma$.  Observe also that the geodesic flow $g^t$ descends to the action on $\GER$ given by $g^t (\xi, \eta, s) = (\xi, \eta, s + t)$, hence this is clearly an action by homeomorphisms.  We also have the following complementary result.

\begin{prop}\label{homeos2}
The $\Gamma$-action on $SX$ descends to an action on $\GER$ by homeomorphisms.
\end{prop}
\begin{proof}
Fix $x \in X$.  First compute
\begin{align*}
\pi_x (\gamma v) &= \left( \gamma v^-, \gamma v^+, b_{\gamma v^-} (\gamma v(0), x) \right) \\
&= \left( \gamma v^-, \gamma v^+, b_{v^-} (v(0), \gamma^{-1} x) \right) \\
&= \left( \gamma v^-, \gamma v^+, b_{v^-} (v(0), x) + b_{v^-} (x, \gamma^{-1} x) \right).
\end{align*}
Now recall $\xi_n \to \xi$ in $\bd X$ if and only if $b_{\xi_n} (\cdot, p) \to b_{\xi} (\cdot, p)$ uniformly on compact subsets, for $p \in X$ arbitrary.  Hence if $v_n^- \to v^-$ then $b_{v_n^-} (x, \gamma^{-1} x) \to b_{v^-} (x, \gamma^{-1} x)$.  So suppose $v_n \to v$ weakly in $SX$.  Then
\begin{align*}
\pi_x (\gamma v)
&= \left( \gamma v^-, \gamma v^+, b_{v^-} (v(0), x) + b_{v^-} (x, \gamma^{-1} x) \right) \\
&= \lim_{n \to \infty} \left( \gamma v_n^-, \gamma v_n^+, b_{v_n^-} (v_n (0), x) + b_{v_n^-} (x, \gamma^{-1} x) \right) \\
&= \lim_{n \to \infty} \pi_x (\gamma v_n).
\end{align*}
Thus $\gamma$ descends to a continuous map $\GER \to \GER$.  But then $\gamma^{-1}$ also descends to a continuous map, and therefore $\Gamma$ acts by homeomorphisms on $\GER$.
\end{proof}

\section{Recurrence}

We now study some of the basic topological properties of the geodesic flow on $SX$.  We want to study these properties both on $SX$ and its weak product structure $\GER$ from the previous section.

\begin{standing hypothesis}
In this section, let $\Gamma$ be a group acting by isometries on a proper CAT($0$) space $X$.
\end{standing hypothesis}

\begin{definition}\label{rec def}
A geodesic $v \in SX$ is said to \defn{$\Gamma$-accumulate} on $w \in SX$ if there exist sequences $t_n \to +\infty$ and $\gamma_n \in \Gamma$ such that $\gamma_n g^{t_n} (v) \to w$ as $n \to \infty$.  A geodesic $v \in SX$ called \defn{$\Gamma$-recurrent} if it $\Gamma$-accumulates on itself.
\end{definition}

The definition given above describes \defn{forward} $\Gamma$-recurrent geodesics.  A \defn{backward} $\Gamma$-recurrent geodesic is a geodesic $v \in SX$ such that $\flip v$ is forward $\Gamma$-recurrent%
, where $\flip \colon SX \to SX$ is the map given by $(\flip v)(t) \mapsto v(-t)$%
.  We will also sometimes use the terms \defn{weakly} and \defn{strongly}, as in \defref{weak-def}, to specify the convergence in \defref{rec def}.

Recurrence is stronger than nonwandering:

\begin{definition}\label{nw def}
A geodesic $v \in SX$ is called \defn{nonwandering mod $\Gamma$} if there exists sequences $v_n \in SX$, $t_n \to +\infty$, and $\gamma_n \in \Gamma$ such that $\gamma_n g^{t_n} (v_n) \to w$ as $n \to \infty$.
\end{definition}

Note that $v \in SX$ is $\Gamma$-recurrent if and only if its projection onto $\modG{SX}$ is recurrent under the geodesic flow $g^t_\Gamma$ on $\modG{SX}$.  Similarly, $v \in SX$ is nonwandering mod $\Gamma$ if and only if its projection is nonwandering under the geodesic flow $g^t_\Gamma$ on $\modG{SX}$.

Eberlein (\cite{eb72}) proved the following result for manifolds of nonpositive curvature; it describes duality in $\bd X$ in terms of geodesics.

\begin{lemma}[Lemma III.1.1 in \cite{ballmann}]\label{duality-prop}
Suppose $X$ is geodesically complete.  If $v, w \in SX$ and $v^+, w^-$ are $\Gamma$-dual, then there exist $(\gamma_n, t_n, v_n) \in \Gamma \times \R \times SX$ such that $v_n \to v$ and $\gamma_n g^{t_n} v_n \to w$.
\end{lemma}

Thus Eberlein observed (see \cite{eb72} and \cite{eb73}) for manifolds of nonpositive curvature that $v \in SX$ is nonwandering mod $\Gamma$ if and only if $v^-$ and $v^+$ are $\Gamma$-dual.  This fact holds for proper, geodesically complete CAT($0$) spaces as well (%
see the discussion preceding Corollary III.1.4 in \cite{ballmann}).

\begin{cor}\label{nonwandering iff dual endpoints}
Suppose $X$ is geodesically complete.  The geodesic $v \in SX$ is nonwandering mod $\Gamma$ if and only if $v^-$ and $v^+$ are $\Gamma$-dual.
\end{cor}

We recall the situation for rank one CAT($0$) spaces (cf.~ \propref{rank one summary} and \propref{rank one summary 3}):

\begin{prop}\label{rank one summary 2}
Let $\Gamma$ be a group acting properly discontinuously, cocompactly, and isometrically on a proper, geodesically complete \textnormal{CAT($0$)} space $X$.  Suppose $X$ contains a rank one geodesic.  The following are equivalent:
\begin{enumerate}
\item \label{ros2 cond: some rank one axis}
$X$ has a rank one axis.
\item \label{ros2 cond: rank one axes are weakly dense}
The rank one axes of $X$ are weakly dense in $\Reg$.
\item \label{ros2 cond: some rank one geodesic is nonwandering}
Some rank one geodesic of $X$ is nonwandering mod $\Gamma$.
\item \label{ros2 cond: every geodesic is nonwandering}
Every geodesic of $X$ is nonwandering mod $\Gamma$.
\item \label{ros2 cond: dense strong recurrence}
The strongly $\Gamma$-recurrent geodesics of $X$ are dense in $SX$.
\item \label{ros2 cond: some dense orbit in SX mod Gamma}
$SX$ has a strongly dense orbit mod $\Gamma$.
\end{enumerate}
\end{prop}
\begin{proof}
\itemrefstar{ros2 cond: rank one axes are weakly dense}%
$\implies$%
\itemrefstar{ros2 cond: some rank one axis}
and
\itemrefstar{ros2 cond: every geodesic is nonwandering}%
$\implies$%
\itemrefstar{ros2 cond: some rank one geodesic is nonwandering}
are immediate.  By \propref{rank one summary}, $X$ has a rank one axis if and only if Chen and Eberlein's duality condition holds, so
\itemrefstar{ros2 cond: some rank one axis}%
$\iff$%
\itemrefstar{ros2 cond: every geodesic is nonwandering}%
$\iff$%
\itemrefstar{ros2 cond: dense strong recurrence}
by Corollaries III.1.4 and II.1.5 of \cite{ballmann}, and
\itemrefstar{ros2 cond: some rank one axis}%
$\implies$%
\itemrefstar{ros2 cond: some dense orbit in SX mod Gamma}
by Theorem III.2.4 of \cite{ballmann}.  By Lemma III.3.2 of \cite{ballmann}, every rank one geodesic that is nonwandering mod $\Gamma$ is a weak limit of rank one axes; this proves
\itemrefstar{ros2 cond: some rank one geodesic is nonwandering}%
$\implies$%
\itemrefstar{ros2 cond: some rank one axis}
and
\itemrefstar{ros2 cond: every geodesic is nonwandering}%
$\implies$%
\itemrefstar{ros2 cond: rank one axes are weakly dense}.

We now prove
\itemrefstar{ros2 cond: some dense orbit in SX mod Gamma}%
$\implies$%
\itemrefstar{ros2 cond: dense strong recurrence}.
Let $v \in SX$ have dense orbit mod $\Gamma$; by \lemref{blem}, $v \in \Reg$.  If
\itemrefstar{ros2 cond: some rank one axis}
fails, then $v^+$ cannot be isolated in the Tits metric on $\bd X$ by \propref{rank one summary}.  Hence $(v^-, v^+)$ cannot be isolated in $\GE$, so $v$ must be strongly $\Gamma$-recurrent; this concludes the proof.
\end{proof}

Similar statements to those of \propref{rank one summary} and \propref{rank one summary 2} can be made for rank one, non-elementary actions.  We mention three that we will use later.

\begin{prop}[Main Theorem in \cite{ham09}]
\label{hamenstadt}
Let $X$ be a proper \textnormal{CAT($0$)} space under a proper, non-elementary, isometric action by a group $\Gamma$ with a rank one element.  All the following hold:
\begin{enumerate}
\item
$\Lam$ is perfect, and $\Gamma$ acts minimally on $\Lam$.
\item
$\SXL$ has a weakly dense orbit mod $\Gamma$.
\item
The rank one axes are weakly dense in $\SXL$.
\end{enumerate}
\end{prop}

We will work mainly with $\Gamma$-recurrence.  The following standard result illustrates the power of $\Gamma$-recurrence.

\begin{lemma}\label{strong rec}
Let $v \in SX$ be a $\Gamma$-recurrent geodesic.  Then every $w \in SX$ with $w^+ = v^+$ $\Gamma$-accumulates on a geodesic parallel to $v$.
\end{lemma}
\begin{proof}
Since $v$ is $\Gamma$-recurrent, there exist sequences $t_n \to +\infty$ and $\gamma_n \in \Gamma$ such that $\gamma_n g^{t_n} (v) \to v$.  So suppose $w \in SX$ has $w^+ = v^+$.  Since $w^+ = v^+$, the function $t \mapsto d(g^t v, g^t w)$ is bounded on $t \ge 0$ by convexity, hence $\set{\gamma_n g^{t_n} w(0)}$ is bounded, and passing to a subsequence we may assume that $\gamma_n g^{t_n} (w) \to u \in SX$.  But then
\[d(v(s), u(s)) = \lim_{n \to \infty} d(\gamma_n g^{t_n} v(s), \gamma_n g^{t_n} w(s)) = \lim_{t \to \infty} d(g^t v(s), g^t w(s))\]
is independent of $s \in \R$, and thus $u$ is parallel to $v$.
\end{proof}

Inspecting the proof, we see that we have actually shown the following.

\begin{lemma}\label{strong acc}
Suppose $v, w \in SX$ have $v^+ = w^+$.  If $v$ $\Gamma$-accumulates on $u \in SX$, then $w$ must $\Gamma$-accumulate on a geodesic parallel to $u$.
\end{lemma}

We will need to deal with weak $\Gamma$-recurrence, so we revisit \lemref{strong rec}.

\begin{lemma}\label{weak rec}
Let $v \in \Reg$ be a weakly $\Gamma$-recurrent geodesic.  Then every $w \in SX$ with $w^+ = v^+$ strongly $\Gamma$-accumulates on a geodesic $u \sim v$.
\end{lemma}
\begin{proof}
By \lemref{quotient}, $v$ strongly $\Gamma$-accumulates on some $u \sim v$.  By \lemref{strong acc}, $w$ must strongly $\Gamma$-accumulate on some $u' \parallel u$.  But $g^t u' \sim u$ for some $t \in \R$, so we may assume $u' \sim u$.
\end{proof}

Since convergence preserves distances between all geodesics $w' \parallel w \in H^s (v)$, by passing to a subsequence we expect convergence of $X_w$ to an isometric embedding into $X_v$.  This is shown in the following lemma.

\begin{lemma}\label{strong acc transversals}
Suppose $w \in SX$ strongly $\Gamma$-accumulates on $v \in SX$.  Then there are isometric embeddings $X_w \into X_v$ and $Y_w \into Y_v$, each of which maps $w(0) \mapsto v(0)$.
\end{lemma}
\begin{proof}
Let $(t_n, \gamma_n) \subset \R \times \Gamma$ be a sequence such that $\gamma_n g^{t_n} w \to v$ in $SX$.  Then, in particular, $\gamma_n g^{t_n} w(0) \to v(0)$ in $X$.  So by the Arzel\`a-Ascoli Theorem, we may pass to a further subsequence such that the natural isometries $Y_w \to \gamma_n Y_{g^{t_n} w}$ converge uniformly to an isometric embedding $\varphi$ of $Y_w$ into $X$.  Since $\gamma_n g^{t_n} (w) \to v$, the map $\varphi$ must extend to an isometric embedding of $X_w$ into $X_v$.  But $\varphi$ must also isometrically embed $Y_w$ into $Y_v$ because $\gamma_n g^{t_n} w(0) \to v(0)$.
\end{proof}

\begin{cor}\label{weak rec transversals}
Let $v \in \Reg$ be weakly $\Gamma$-recurrent.  Then for every $w \in SX$ with $w^+ = v^+$, there are isometric embeddings $X_w \into X_v$ and $Y_w \into Y_v$.
\end{cor}
\begin{proof}
By \lemref{weak rec}, $w$ strongly $\Gamma$-accumulates on a geodesic $u \sim v$.  Since $u \parallel v$, we have $X_u = X_v$ and $Y_u = Y_v$.  Now apply \lemref{strong acc transversals}.
\end{proof}

The proof of the next lemma combines a few standard arguments about CAT($0$) spaces.  It demonstrates that weakly $\Gamma$-recurrent rank one geodesics share some important properties with rank one axes, which have both endpoints isolated in the Tits metric.  %

\begin{lemma}\label{rec => Tits-isolated}
If $v \in \Reg$ is weakly $\Gamma$-recurrent, then $v^+$ is isolated in the Tits metric---that is, $v^+$ has infinite Tits distance to every other point in $\bd X$.
\end{lemma}
\begin{proof}
Let $v \in \Reg$ be weakly $\Gamma$-recurrent.  By \lemref{quotient}, there is a sequence $(t_n, \gamma_n)$ in $\R \times \Gamma$ with $t_n \to +\infty$ and $u \sim v$ such that $\gamma_n g^{t_n} (v) \to u$ strongly; note $u^+ = v^+$.  Let $p = u(0)$ and $p_n = v(t_n)$.  Suppose $\xi \in \bd X$ has $d_T (\xi, v^+) < \pi$; in particular, $\angle (\xi, v^+) < \pi$.  Passing to a subsequence, we may assume $\gamma_n \xi \to \eta \in \bd X$.  Clearly $\gamma_n p_n \to p$, hence $\angle_{p} (\eta, v^+) \ge \limsup_{n \to \infty} \angle_{\gamma_n p_n} (\gamma_n \xi, \gamma_n v^+)$ by upper semicontinuity (see \cite[Proposition 9.2(2)]{bridson}).  But $\angle_{p_n} (\xi, v^+) \to \angle (\xi, v^+)$ because $p_n = v(t_n)$ (see \cite[Proposition 9.8(2)]{bridson} or \cite[Proposition II.4.2]{ballmann}).  And $\gamma_n v^+ \to v^+$, so $\angle (\eta, v^+) \le \liminf_{n \to \infty} \angle (\gamma_n \xi, \gamma_n v^+)$ by lower semicontinuity (see \cite[Proposition 9.5(2)]{bridson} or \cite[Proposition II.4.1]{ballmann}).  Thus we have
$\angle_{p} (\eta, v^+) \ge \limsup_{n \to \infty} \angle_{\gamma_n p_n} (\gamma_n \xi, \gamma_n v^+) = \limsup_{n \to \infty} \angle_{p_n} (\xi, v^+) = \angle (\xi, v^+) = \liminf_{n \to \infty} \angle (\gamma_n \xi, \gamma_n v^+) \ge \angle (\eta, v^+)$.
But then $\angle_{p} (\eta, v^+) = \angle (\eta, v^+)$ by definition of $\angle$, and so there is a flat sector bounded by $(p, \eta, v^+)$ (see \cite[Corollary II.9.9]{bridson} or \cite[Proposition II.4.2]{ballmann}).

Now the points $p_n = v(t_n)$ lie in arbitrarily large balls of a flat half-plane bounded by the image of $v$.  By the Arzel\`a-Ascoli theorem, $p = \lim \gamma_n p_n$ lies on a full flat half-plane bounded by the image of $u = \lim \gamma_n v$.  But this contradicts the fact that $u \sim v \in \Reg$.
\end{proof}

\section{Bowen-Margulis Measures}

We now construct our first Bowen-Margulis measures.  In this section, we put them on the weak product structure $\GER$ and its quotient under $\Gamma$.  Near the end of \secref{props of B-M measures}, we will finally be able to define Bowen-Margulis measures on $SX$ and its quotient under $\Gamma$.

\begin{standing hypothesis}
In this section, let $X$ be a proper CAT($0$) space under a proper, non-elementary, isometric action by a group $\Gamma$.
\end{standing hypothesis}

Assuming $\double{\mu_x}$ gives positive measure to $\REL$, \lemref{beta is continuous} allows us to define a Borel measure $\mu$ on $\dbL$ by
\[d\mu (\xi, \eta)
= e^{-\delta_\Gamma \beta_x (\xi, \eta)} \chi_{\REL} (\xi, \eta) d\mu_x (\xi) d\mu_x (\eta),\]
where $x \in X$ is arbitrary.
It follows easily from the definitions that one has
\[d\mu (\xi, \eta)
= e^{-\delta_\Gamma \beta_p (\xi, \eta)} \chi_{\REL} (\xi, \eta) d\mu_p (\xi) d\mu_p (\eta)\]
for all $p \in X$.  Thus $\mu$ does not depend on choice of $x \in X$ and is $\Gamma$-invariant.  However, we still need to show $(\double{\mu_x}) (\REL) > 0$ in order for $\mu$ to be a measure.

\begin{lemma}\label{full-mu}
If $\Gamma$ has a rank one element, then $\mu$ is a measure and $\supp(\mu) = \dbL$.
\end{lemma}
\begin{proof}
By \propref{hamenstadt}, $\Gamma$ acts minimally on $\Lam$, and so $\supp(\mu_x) = \Lam$.  Hence any nonempty open set in $\dbL$ has positive $(\double{\mu_x})$-measure, and therefore $(\double{\mu_x}) (\REL) > 0$.  Thus $\mu$ is a measure.  Moreover, $\REL \subseteq \supp(\mu)$.  By \propref{hamenstadt}, %
$\REL$ is dense in $\dbL$, so $\supp(\mu) = \dbL$.
\end{proof}

The above proof also shows:

\begin{cor}\label{mu is a measure}
If $\supp(\mu_x) = \bd X$ and $\RE$ is nonempty, then $\mu$ is a measure and $\RE \subseteq \supp(\mu)$.
\end{cor}

\begin{rem}
We will show (\corref{RE has full measure}) that if $\Gamma$ has a rank one element, then the $\chi_{\REL}$ term can be removed from the definition of $\mu$.
\end{rem}

We want to use $\mu$ to create a $\Gamma$-invariant Borel measure on $SX$.  Potentially, one might do so on $\Reg_C$ directly, but it is not clear how to ensure that the result would be Borel.  We can do so on the related space $\GER$, however.

\begin{definition}
Suppose $\supp (\double{\mu_x}) (\REL) > 0$.  The \defn{Bowen-Margulis measure} $m$ on $\GER$ is given by $m = \mu \times \Leb$, where $\Leb$ is Lebesgue measure on $\R$.
\end{definition}

Now $\Gamma$ is a countable group acting properly (by \lemref{blem}) and by homeomorphisms (by \propref{homeos2}) on $\RER$ (which admits a proper metric)%
, preserving the Borel measure $m$.
Thus there is (see, for instance, Appendix A of \cite{ricks-thesis}) a unique Borel quotient measure $m_{\Gamma}$ on $\modG{(\RER)}$ satisfying the following characterizing property:
\vspace{-8pt}
\begin{quotation}
\begin{equation}\tag{$\dagger$}
\label{qm cond for B-M measure}
\int_A h \, dm
  = \int_{\modG{(\RER)}} (\bar h \cdot \bar f_A) \,dm_{\Gamma}
\end{equation}
for all Borel sets $A \subseteq \RER$ and $\Gamma$-invariant Borel maps $h \colon \RER \to [0, \infty]$, where $f_A (v) = \abs{\setp{\gamma \in \Gamma}{\gamma v \in A}}$ and $\bar h, \bar f_A$ are the induced maps on $\modG{(\RER)}$.
\end{quotation}
Moreover, the geodesic flow $g^t_\Gamma$ on $\RER$ preserves $m_\Gamma$.

The measure $m_{\Gamma}$ on $\modG{(\RER)}$ naturally extends to a Borel measure on $\modG{(\GER)}$, also denoted $m_{\Gamma}$.  This measure is called the \defn{Bowen-Margulis measure} on $\modG{(\GER)}$.

Note that by \itemrefstar{qm cond for B-M measure}, a Borel set $A \subset \GER$ has $m(A) = 0$ if and only if $m_\Gamma (\pr(A)) = 0$, where $\pr \colon \GER \to \modG{(\GER)}$ is the canonical projection.  In particular, $\supp(m_\Gamma) = \modG{(\GELR)}$.

\begin{prop}\label{Bowen-Margulis quotient measure}
If $\Gamma$ acts cocompactly on $X$, then $m_{\Gamma}$ is finite.
\end{prop}
\begin{proof}
It suffices to show $m(F) < \infty$ for some $F \subseteq \GER$ such that $\Gamma F = \GER$.  Now the $\Gamma$-action on $SX$ is cocompact by \lemref{SX is proper}, so there is a compact $K \subset SX$ such that $\Gamma K = SX$.  Let $x \in X$ and $F = \pi_x (K)$.  Then $\Gamma F = \GER$ because $\Gamma K = SX$.  We will show $m(F) < \infty$.

Since $F$ is compact by continuity of $\pi_x$, we have $F \subseteq \GE \times [-r, r]$ for some finite $r \ge 0$; thus it suffices to prove $\mu(\emap(K)) < \infty$.  Let $A = \setp{v(0) \in X}{v \in K}$.  By \lemref{beta}, $\beta_x (v) = (b_{\xi} + b_{\eta})(v(0), x)$.  Hence
\[\beta_x (K) \subseteq \setp{(b_{\xi} + b_{\eta}) (p, x)}{(\xi, \eta) \in \emap(K) \text{ and } p \in A}.\]
So $\abs{\beta_x (K)} \le 2R$, where $R$ is the diameter of $A$ in $X$, because the map $p \mapsto b_{\zeta} (p, x)$ is $1$-Lipschitz for all $\zeta \in \bd X$.  Thus
\begin{align*}
\mu(\emap(K)) &= \int_{\emap(K) \cap \REL} e^{-\delta_\Gamma \beta_x (\xi, \eta)} d\mu_x (\xi) d\mu_x (\eta)
\le \int_{\emap(K) \cap \REL} e^{\delta_\Gamma \cdot 2R} d\mu_x (\xi) d\mu_x (\eta)
\le e^{\delta_\Gamma \cdot 2R}.
\end{align*}
Hence $\mu(\emap(K)) < \infty$, and therefore $m(F) < \infty$.  Thus $m_\Gamma$ is finite.
\end{proof}

The following lemma is a simple consequence of Poincar\'e recurrence.

\begin{lemma}\label{recurrence}
Suppose $m_{\Gamma}$ is finite.  Let $W$ be the set of $w \in SX$ such that both $w$ and $\flip w$ are weakly $\Gamma$-recurrent.  Then $\mu(\emap(SX \setminus W)) = 0$.
\end{lemma}
\begin{proof}
Note that $\modG{(\GER)}$ has a countable basis, so by Poincar\'e recurrence, the set $W_\Gamma$ of forward and backward recurrent points in $\modG{(\GER)}$ has full $m_\Gamma$-measure.  Now $W$ is $\Gamma$-invariant and projects down to $W_\Gamma$ in $\modG{(\GER)}$, so $m ((\GER) \setminus \pi_x (W)) = 0$.  The result follows from $g^t$-invariance of $W$.
\end{proof}

We conclude this section by extending \propref{rank one summary} and \propref{rank one summary 2}.

\begin{prop}\label{rank one summary 3}
Let $\Gamma$ be a group acting properly discontinuously, cocompactly, and isometrically on a proper, geodesically complete \textnormal{CAT($0$)} space $X$.  Suppose $X$ contains a rank one geodesic.  The following are equivalent:
\begin{enumerate}
\item \label{ros3 cond: some rank one axis}
$X$ has a rank one axis.
\item \label{ros3 cond: mu_x has full support}
$\supp(\mu_x) = \bd X$.
\item \label{ros3 cond: RE has positive measure}
$(\mu_x \times \mu_x)(\RE) > 0$.
\item \label{ros3 cond: some weak recurrence}
Some rank one geodesic of $X$ is weakly $\Gamma$-recurrent.
\end{enumerate}
\end{prop}
\begin{proof}
\itemrefstar{ros3 cond: mu_x has full support}%
$\implies$%
\itemrefstar{ros3 cond: RE has positive measure} is clear because $\RE$ is open;
\itemrefstar{ros3 cond: RE has positive measure}%
$\implies$%
\itemrefstar{ros3 cond: some weak recurrence} is a corollary of \lemref{recurrence}.  For
\itemrefstar{ros3 cond: some rank one axis}%
$\implies$%
\itemrefstar{ros3 cond: mu_x has full support}, recall (\propref{rank one summary}) that the $\Gamma$-action on $\bd X$ is minimal if $X$ has a rank one axis; the claim follows immediately.

We now prove
\itemrefstar{ros3 cond: some weak recurrence}%
$\implies$%
\itemrefstar{ros3 cond: some rank one axis}.
Suppose $v \in \Reg$ is weakly $\Gamma$-recurrent; we may assume $v \in \Reg_C$.  By \lemref{quotient}, we may find $\gamma_n g^{t_n} (v) \to u \sim v$, and the natural isometries $Y_v \to \gamma_n Y_{g^{t_n} v}$ converge uniformly (on compact subsets) to an isometric embedding $\varphi$ of $Y_v$ into $Y_u = Y_v$.  But $v(0)$ is the circumcenter of $Y_v$, and that is isometry-invariant, so we must have $u = v$.  Thus $v$ is strongly $\Gamma$-recurrent, and therefore nonwandering mod $\Gamma$.  Therefore, $X$ has a rank one axis by \propref{rank one summary 2}.
\end{proof}

\section{Properties of Bowen-Margulis Measures}
\label{props of B-M measures}

We now are in a position to prove some important properties about the Bowen-Margulis measures we constructed on $\GER$ and $\modG{(\GER)}$.  In \thmref{isolated almost everywhere}, we use the Bowen-Margulis measures to obtain a structural result about the Patterson-Sullivan measures.  Then (\thmref{a.e.-geodesic-is-lonely}) we prove a structural result about $SX$.  This theorem allows us to finally define Bowen-Margulis measures on $SX$ and $\modG{SX}$.  We end the section by showing that the geodesic flow is ergodic with respect to the Bowen-Margulis measure on $\modG{SX}$.

\begin{standing hypothesis}
In this section, let $X$ be a proper CAT($0$) space under a proper, non-elementary, isometric action by a group $\Gamma$ with a rank one element.  Assume that $m_{\Gamma}$ is finite.
\end{standing hypothesis}

By \lemref{recurrence}, we have weak recurrence almost everywhere.  Our next theorem uses \lemref{rec => Tits-isolated} to capitalize on the prevalence of recurrence.

\begin{theorem}[\thmref{main isolated}]
\label{isolated almost everywhere}
\stateisolated
\end{theorem}
\begin{proof}
Let $\Omega$ be the set of Tits-isolated points in $\bd X$, and let $\xi \in \bd X$.  Find $v \in \Reg$ an axis of a rank one element; we may assume $v^- \neq \xi$.  Then by \lemref{ping-pong}, there is a geodesic $w \in \Reg$ with $(w^-, w^+) = (v^-, \xi)$.  By \lemref{blem}, we have an open product neighborhood $U \times V$ of $(v^-, \xi)$ in $\RE$.

Let $W$ be the set of weakly $\Gamma$-recurrent geodesics in $SX$.  Then $\mu ((U \times V) \setminus \emap(W)) = 0$ by \lemref{recurrence}.  So by Fubini's theorem, there exists $W^+ \subseteq V$ such that $\mu_x (V \setminus W^+) = 0$, and $\mu_x (\setp{\zeta \in U}{(\zeta, \eta) \notin \emap(W)}) = 0$ for every $\eta \in W^+$.  Now by \lemref{rec => Tits-isolated}, if $v \in \Reg$ is weakly $\Gamma$-recurrent, then $v^+$ is Tits-isolated.  Hence $W^+ \subseteq \Omega$.

Thus we have shown that every $\xi \in \bd X$ has a neighborhood $V$ such that $\mu_x (V \setminus \Omega) = 0$.  The theorem follows by compactness of $\bd X$.
\end{proof}

Since $w^+$ is a radial limit point for any weakly $\Gamma$-recurrent rank one geodesic $w$, the above proof has the following corollary.

\begin{cor}
\mae{$\mu_x$} $\xi \in \bd X$ is a radial limit point.
\end{cor}

\begin{cor}\label{RE has full measure}
$(\double{\mu_x}) (\dbX \setminus \REL) = 0$.
\end{cor}
\begin{proof}
Let $\xi \in \bd X$ be a Tits-isolated radial limit point.  Then $(\xi, \eta) \in \RE$ for all $\eta \in \bd X \setminus \set{\xi}$.  Since $\mu_x (\set{\xi}) = 0$ by \corref{radial limits}, we see that \mae{$\mu_x$} $\eta \in \bd X$ has $(\xi, \eta) \in \REL$.  The result follows from \thmref{isolated almost everywhere} and Fubini's theorem.
\end{proof}

By \corref{RE has full measure}, $\mu$ and $\double{\mu_x}$ are in the same measure class (that is, each is absolutely continuous with respect to the other), so $\mu (\GE \setminus \RE) = 0$.  Thus almost no geodesic in $X$ (with respect to the Bowen-Margulis measure $m$ on $\GER$) bounds a flat half-plane.

Our next goal (\thmref{a.e.-geodesic-is-lonely}) is to show that almost no geodesic in $X$ bounds a flat strip of any width---that is, $\diam Y_v = 0$ for almost every geodesic $v$.  We will need a few lemmas, the first of which describes the upper semicontinuity property of the map $v \mapsto Y_v$ from $SX$ into the space of closed subsets of $X$ (with the Hausdorff metric).

\begin{lemma}\label{upper semicontinuity of Y_v}
If a sequence $(v_n) \subset SX$ converges to $v \in \Reg$ then some subsequence of $(Y_{v_n})$ converges, in the Hausdorff metric, to a closed subset $A$ of $Y_v$.
\end{lemma}
\begin{proof}
Let $R$ be the diameter of $Y_v$.  By \lemref{SX is proper}, the closed ball $B$ in $SX$ about $v$ of radius $2R$ is compact, so the space $\mathcal C B$ of closed subsets of $B$ is compact under the Hausdorff metric.  For $w \in SX$, let $P'_{w} = \setp{u \in SX}{u \sim w}$.  Eventually every $P'_{v_n}$ lies in $B$, so some subsequence $( P'_{v_{n_k}} )$ converges in $\mathcal C B$.  But every limit point of $w_n \in P'_{v_n}$ must lie in $P'_v$, thus $( Y_{v_{n_k}} )$ converges, in the Hausdorff metric, to a closed subset $A$ of $Y_v$.
\end{proof}

\begin{lemma}\label{constant a.e. 2}
Suppose $\psi \colon \RE \to S$ is a function from $\RE$ to a set $S$.  If $\Omega$ is a set of full $\mu$-measure in $\RE$ such that $\psi((a, b)) = \psi((a, d)) = \psi((c, d))$ for any $(a, b), (a,d), (c, d) \in \Omega$, then $\psi$ is constant $\mu$-a.e.~ on $\RE$.
\end{lemma}
\begin{proof}
By Fubini's theorem, there exists a subset $A$ of $\Lam$ such that $\mu_x (\Lam \setminus A) = 0$ and every $a \in A$ has $(a, b) \in \Omega$ for $\mu_x$-a.e.~ $b \in \Lam$.  Let $(a, b) \in (A \times \Lam) \cap \Omega$; by choice of $A$ there is some $B \subseteq \Lam$ such that $\mu_x (\Lam \setminus B) = 0$ and $\set{a} \times B \subset \Omega$.  So take any $(c, d) \in (A \times B) \cap \Omega$; then $(a, d) \in (A \times B) \cap \Omega$ by choice of $B$, so $(c, d), (a, d), (a, b) \in \Omega$.  Hence $\psi((c, d)) = \psi((a, d)) = \psi((a, b))$ by hypothesis.  Thus $\psi$ is constant across $(A \times B) \cap \Omega$, which has full measure in $\Lam \times \Lam$.  Apply Fubini's theorem again.
\end{proof}

\begin{rem}
The function $\psi$ in \lemref{constant a.e. 2} is not required to be measurable.  It suffices for $\psi$ to be constant on a set of full measure.
\end{rem}

\begin{lemma}\label{unique-isometry-type}
The isometry type of $Y_v$ is the same for $\mu$-a.e.~ $(v^-, v^+) \in \RE$.
\end{lemma}
\begin{proof}
Let $W$ be the set of $w \in \Reg$ such that $w$ and $\flip w$ are both weakly $\Gamma$-recurrent.  If $u, v \in W$ have $u^+ = v^+$ or $u^- = v^-$, then by \corref{weak rec transversals}, we have isometric embeddings between the compact metric spaces $Y_u$ and $Y_v$, and therefore $Y_u$ and $Y_v$ are isometric (see \cite[Theorem 1.6.14]{burago}).  Thus we may apply \lemref{constant a.e. 2} to the map $\psi$ taking $c \in \RE$ to the isometry type of $Y_c$, with $\Omega = \emap(W)$ by \lemref{recurrence}.
\end{proof}

Let $\Zerowidth = \setp{v \in SX}{\diam (Y_v) = 0}$, the set of zero-width geodesics.  Let $\ZE = \emap (\Zerowidth)$, the set of $(\xi, \eta) \in \GE$ such that no $v \in SX$ with $(v^-, v^+) = (\xi, \eta)$ bounds a flat strip of positive width.  By semicontinuity of the map $v \mapsto Y_v$ (\lemref{upper semicontinuity of Y_v}), the width function $v \mapsto \diam (Y_v)$ is semicontinuous on $SX$.  Thus $\ZE \subseteq \RE$ is Borel ($\emap \colon \Reg \to \RE$ being a topological quotient map by \lemref{quotient}).  Let
\[\ZL = \Zerowidth \cap E^{-1} (\dbL)
\quad \text{and} \quad
\ZEL = \ZE \cap (\dbL).\]

\begin{lemma}\label{ZL is dense}
If $w \in \SXL$ has stongly dense orbit in $\SXL$ mod $\Gamma$, then $w \in \ZL$.
\end{lemma}
\begin{proof}
By \lemref{strong acc transversals}, every $v \sim w$ induces an isometric embedding $Y_{w} \into Y_v$ that maps $w(0) \mapsto v(0)$.  Since $Y_v = Y_w$, this map is an isometry (by \cite[Theorem 1.6.14]{burago}), and therefore $\Isom(Y_w)$ acts transitively on $Y_w$.  But the circumcenter of $Y_w$ is an isometry invariant (note $w \in \Reg$), hence $Y_w$ must be a single point.  Thus $w \in \ZL$.
\end{proof}

\begin{standing hypothesis}
For the remainder of this section, assume $\ZL$ is nonempty.
\end{standing hypothesis}

\begin{theorem}[\thmref{main lonely}]
\label{a.e.-geodesic-is-lonely}
\statelonely
\end{theorem}
\begin{proof}
Let $\Specialset \subseteq \RL$ be the preimage under $\emap$ of the a.e.-set in $\REL$ from \lemref{unique-isometry-type}; then $\pi_x (\Specialset)$ has full $m$-measure.  Since $\Specialset$ is weakly dense in $\RL$, by \lemref{quotient} and semicontinuity of the map $v \mapsto Y_v$ there is an isometric embedding $Y_u \into Y_v$ for every $u \in \Specialset$ and $v \in \RL$.  Thus $\ZL \neq \varnothing$ gives us $\Specialset \subseteq \ZL$, and therefore $\ZEL$ has full $\mu$-measure.
\end{proof}

\begin{lemma}\label{weak things become strong}
If $v_n \to v$ weakly, and $v \in \Zerowidth$, then $v_n \to v$ strongly.
\end{lemma}
\begin{proof}
Let $v \in \Zerowidth$, and suppose $v_n \to v$.  Take an arbitrary subsequence of $(v_n)$.  By \lemref{quotient}, there is a further subsequence that converges strongly to some $u \sim v$.  By hypothesis on $v$, we get $u = v$.  Thus we have shown that every subsequence of $(v_n)$ contains a further subsequence that converges strongly to $v$.  Therefore, $v_n \to v$ strongly.
\end{proof}

\begin{cor}\label{strong stable sets}
For \mae{$m$} $v \in \SXL$,
\begin{align*}
H^s (v) &= \setp{w \in SX}{d(v(t), w(t)) \to 0 \text{ as } t \to +\infty} \\
\text{and}\quad
H^u (v) &= \setp{w \in SX}{d(v(t), w(t)) \to 0 \text{ as } t \to -\infty}.
\end{align*}
\end{cor}
\begin{proof}
By \thmref{a.e.-geodesic-is-lonely} and \lemref{recurrence}, \mae{$m$} $v \in \SXL$ is zero-width and weakly $\Gamma$-recurrent.  For such $v$, by \lemref{strong rec} every $w \in H^s (v)$ has $d(v(t), w(t)) \to 0$ as $t \to +\infty$, which establishes the claim for $H^s (v)$.  A symmetric argument establishes the claim for $H^u (v)$.
\end{proof}

\begin{cor}\label{homeo on Specialset}
The restriction of $\pi_x$ to $\Zerowidth$ is a homeomorphism onto its image.
\end{cor}
\begin{proof}
Fix $x \in X$.  By definition, $\pi_x \res{\Zerowidth}$ is injective, hence bijective onto its image.  Since $\pi_x$ is continuous (that is, every strongly convergent sequence in $\Zerowidth$ is weakly convergent), it remains to observe that $\pi_x \res{\Zerowidth}^{-1}$ is continuous (that is, every weakly convergent sequence in $\Zerowidth$ is strongly convergent) by \lemref{weak things become strong}.
\end{proof}

\begin{definition}
By \corref{homeo on Specialset}, $\pi_x \res{\Zerowidth}$ maps Borel sets to Borel sets, hence we may view $m$ as a $g^t$- and $\Gamma$-invariant Borel measure on $SX$ by setting $m(A) = m(\pi_x (A \cap \Zerowidth))$ for any Borel set $A \subseteq SX$.  We will write $m$ for this measure on $SX$, and $m_\Gamma$ for the corresponding finite Borel measure on $\modG{SX}$.
\end{definition}

\begin{prop}\label{full support on SX}
Suppose $\ZL$ is strongly dense in $\SXL$.  Then the Bowen-Margulis measure $m$ on $SX$ has full support on $\SXL$.
\end{prop}
\begin{proof}
For clarity, we write $m_{down}$ for the measure $m$ on $\GER$ and $m_{up}$ for the measure $m$ on $SX$ defined by $m_{up} (A) = m_{down} (\pi_x (A \cap \Zerowidth))$ for all Borel sets $A \subseteq SX$.  Our goal is to show that $\supp(m_{up}) = \SXL$.

We claim that $\overline{\ZL} \subseteq \supp(m_{up})$.  Since $\supp(m_{up})$ is closed, it suffices to show that $\ZL \subseteq \supp(m_{up})$.  So let $v \in \ZL$ and let $U \subseteq SX$ be an open set containing $v$.  Then $U \cap \Zerowidth$ is open in $\Zerowidth$ by definition, so $\pi_x (U \cap \Zerowidth)$ is open in $\pi_x (\Zerowidth)$ because $\pi_x \res{\Zerowidth}$ is a homeomorphism.  This means $\pi_x (U \cap \Zerowidth) = V \cap \pi_x (\Zerowidth)$ for some open set $V$ of $\GER$.  But $\pi_x (v) \in V$, so $V$ is nonempty.  Since $m_{down}$ has full support on $\RELR$ (\lemref{full-mu}), we have $m_{down} (V) > 0$.  But $\pi_x (\Zerowidth)$ has full measure in $\GER$, and thus
\[m_{up} (U)
= m_{down} (\pi_x (U \cap \Zerowidth))
= m_{down} (V \cap \pi_x (\Zerowidth))
= m_{down} (V)
> 0.\]
Hence $v \in \supp(m_{up})$, as claimed.

Thus $\SXL \subseteq \overline{\ZL} \subseteq \supp(m_{up}) \subseteq \SXL$, and therefore $\supp(m_{up}) = \SXL$.
\end{proof}

\begin{prop}\label{strong density}
Let $X$ be a proper \textnormal{CAT($0$)} space under a proper, non-elementary, isometric action by a group $\Gamma$ with a rank one element.  Suppose $m_{\Gamma}$ is finite.  The following are equivalent:
\begin{enumerate}
\item \label{strongly dense orbit}
$\SXL$ has a strongly dense orbit mod $\Gamma$.
\item \label{ZL strongly dense}
$\ZL$ is strongly dense in $\SXL$.
\item \label{full support for m}
$\ZL$ is nonempty and $\supp(m) = \SXL$.
\end{enumerate}
If $X$ is geodesically complete then all the above conditions hold.
And if any of the above conditions holds, the rank one axes are strongly dense in $\SXL$.
\end{prop}
\begin{proof}
Note
\itemrefstar{strongly dense orbit}%
$\implies$%
\itemrefstar{ZL strongly dense}
by \lemref{ZL is dense},
\itemrefstar{ZL strongly dense}%
$\implies$%
\itemrefstar{full support for m}
by \propref{full support on SX}, and
\itemrefstar{full support for m}%
$\implies$%
\itemrefstar{ZL strongly dense}
trivially.  For
\itemrefstar{ZL strongly dense}%
$\implies$%
\itemrefstar{strongly dense orbit},
let $w_0 \in \RL$ have weakly dense orbit in $\RL$ mod $\Gamma$; by
\itemrefstar{ZL strongly dense},
it suffices to prove that $w_0$ strongly $\Gamma$-accumulates on every $v \in \ZL$.  But this follows from \lemref{weak things become strong} because $w_0$ weakly $\Gamma$-accumulates on $v$ by choice of $w_0$.

The statement about geodesic completeness comes from observing that the proof of Theorem III.2.4 of \cite{ballmann} extends.  For the statement about rank one axes:
By
\itemrefstar{ZL strongly dense},
it suffices to prove that every $v \in \ZL$ is a strong limit of rank one axes.  But this follows from \lemref{weak things become strong} because the rank one axes of $X$ are weakly dense in $\RL$ by \propref{hamenstadt}.
\end{proof}

\begin{rem}
Geodesic completeness of $X$ is a strictly stronger condition than the equivalent conditions in \propref{strong density}, and strong density of rank one axes in $\SXL$ is strictly weaker, as the following two examples show.
\end{rem}

\begin{example}
Glue a flat strip isometrically along a closed geodesic in a closed, negatively curved manifold.  Call the universal cover of the resulting space $X$, and let $\Gamma$ be the associated group of deck transformations.  Here the rank one axes are strongly dense in $SX$, and $\ZL = \Zerowidth$ is weakly dense in $\SXL = SX$, but $\ZL$ is not strongly dense in $\SXL$.
\end{example}

\begin{example}
Take a flat cylinder $\mS^1 \times [0, 1]$, and attach a circle to each of $(p, 0)$ and $(p, 1)$ for some $p \in \mS^1$.  Call the universal cover of resulting space $X$, and let $\Gamma$ be the associated group of deck transformations ($\Gamma \cong \Z * \Z * \Z$).  Then $\ZL = \Zerowidth$ is strongly dense in $\SXL = SX$, but $X$ is not geodesically complete.
\end{example}

\thmref{main cocompactness} follows from Propositions \ref{Bowen-Margulis quotient measure} and \ref{strong density}:

\begin{theorem}[\thmref{main cocompactness}]
\label{cocompactness results}
\statecocompactness
\end{theorem}

\section{On Links}

It is convenient here to recall a few properties of links in CAT($\kappa$) spaces.  CAT($\kappa$) spaces, like CAT($0$) spaces, satisfy a triangle comparison requirement for small triangles, but the comparison is to a triangle in a complete, simply connected manifold of constant curvature $\kappa$.  One may always put $\kappa = 0$ in this section, which will be the only case we use later in this paper.

We will begin with the definition of a link, and then give a proof of \propref{geodesic completeness}.  Lytchak (\cite{lytchak05}) states a version of this result when $Y$ is CAT($1$) and compact, but we need to allow $Y$ to be proper in place of compact.

\begin{definition}
Let $Y$ be a CAT($\kappa$) space and $p \in Y$.  Write $\Sp{p}{Y}$ for the space of geodesic germs in $Y$ issuing from $p$, equipped with the metric $\angle_p$, and write $\link{p}{Y}$ (called the \defn{link of $p$} or the \defn{link of $Y$ at $p$}) for the completion of $\Sp{p}{Y}$ (cf.~ \cite{bridson} or \cite{lytchak05}).
\end{definition}

By Nikolaev's theorem (Theorem II.3.19 in \cite{bridson}), the link of $p \in Y$ is CAT($1$).  Note also that if $Y$ is proper and geodesically complete, $\Sp{p}{Y}$ is already complete.

\begin{definition}
Let $Y$ be a \textnormal{CAT($\kappa$)} space and $p \in Y$.  The \defn{tangent cone at $p$}, denoted $T_p Y$, is the Euclidean cone on $\link{p}{Y}$, the link of $p$.
\end{definition}

\begin{lemma}\label{Gromov-Hausdorff limits of geodesics}
Let $Y$ be compact and the Gromov-Hausdorff limit of a sequence $(Y_i)$ of compact metric spaces.  Then any sequence of isometric embeddings $\sigma_i \colon [0,1] \to Y_i$ has a subsequence that converges to an isometric embedding $\sigma \colon [0,1] \to Y$.
\end{lemma}
\begin{proof}
A subsequence of the spaces $Y_i$, together with $Y$, may be isometrically embedded into a single compact metric space.  Apply the Arzel\`a-Ascoli theorem.
\end{proof}

\begin{lemma}[Theorem 9.1.48 in \cite{burago}]
\label{tangent cone is a Gromov-Hausdorff limit}
Let $(Y, d)$ be a proper, geodesically complete \textnormal{CAT($\kappa$)} space, $\kappa \in \R$, and let $p \in Y$.  Fix $r > 0$, and for $t \in (0,1]$, let $(Y_t, d_t)$ be the compact metric space $(\overline{B}_Y (p, r t), \frac{1}{t} d)$.  Let $(Y_0, d_0)$ be the closed ball of radius $r$ about the cone point $\bar p$ in the tangent cone $T_p Y$ at $p$.  Then $Y_t \to Y_0$ in the Gromov-Hausdorff metric as $t \to 0$.
\end{lemma}

\begin{prop}\label{geodesic completeness}
For any proper \textnormal{CAT($\kappa$)} space $Y$, $\kappa \in \R$, the following are equivalent:
\begin{enumerate}
\item  $Y$ is geodesically complete.
\item  For every point $p \in Y$, the tangent cone $T_p Y$ at $p$ is geodesically complete.
\item  For every point $p \in Y$, the link $\link{p}{Y}$ of $p$ is geodesically complete and has at least two points.
\item  For every point $p \in Y$, every point in the link $\link{p}{Y}$ of $p$ has at least one antipode---that is, for every $\alpha \in \link{p}{Y}$, there is some $\beta \in \link{p}{Y}$ such that $d(\alpha, \beta) \ge \pi$.
\end{enumerate}
\end{prop}
\begin{proof}
The implication $(1) \implies (2)$ is clear from Lemmas \ref{Gromov-Hausdorff limits of geodesics} and \ref{tangent cone is a Gromov-Hausdorff limit}.  And $(2) \implies (3)$ is immediate from the fact that radial projection $T_p Y \to \link{p}{Y}$ is a bijective map on geodesics (see the proof of Proposition I.5.10(1) in \cite{bridson}).  Since $Y$ is CAT($\kappa$), each component of the link $\link{p}{Y}$ of $p$ is CAT($1$) and therefore has no geodesic circles of length $< \pi$; thus $(3) \implies (4)$.

Finally, we prove $(4) \implies (1)$.  Let $r > 0$ be small enough that geodesics in $Y$ of length $< 3r$ are uniquely determined by their endpoints.  It suffices to show that for $p, q \in Y$ with $d(p, q) \le r$, there is some $y_0 \in Y$ such that $d(p, y_0) = r$ and $b_q (y_0, p) = d(y_0, p)$  (recall $b_q (y_0, p) = d(q, y_0) - d(q, p)$).  So let $f_p (y) = b_q (y, p)$.  Let $\delta > 0$, and define $A_{\delta} = \setp{y \in Y}{f_p (y) \ge (1 - \delta) d(y, p)}$.  Since $Y$ is proper and $A_{\delta}$ is closed, $A_{\delta} \cap \cl{B}(p, r)$ compact.  Thus some $y = y_{\delta} \in A_{\delta} \cap \cl{B}(p, r)$ maximizes $f_p$ on $A_{\delta} \cap \cl{B}(p, r)$.  If $d(p, y) < r$ then by $(4)$ and the density of $\Sp{p}{Y}$ in $\link{p}{Y}$, some $z \in Y$ with $d(y, z) \le r - d(p, y)$ satisfies $f_y (z) \ge (1 - \delta) d(z, y) > 0$ (assume $\delta < 1$).  Since $f_p (z) = f_y (z) + f_p (y)$ by the cocyle property of Busemann functions, we obtain $f_p (z) > f_p (y)$; furthermore, $z \in A_{\delta}$ by the triangle inequality.  But this contradicts the maximality of $y$, hence we must have $d(p, y) = r$.  Now take a sequence $\delta_n \to 0$ and let $y_0$ be a limit point of $y_{\delta_n}$.  Then $d(p, y_0) = r$ and $f_p (y_0) = d(y_0, p)$, as required.
\end{proof}

\section{Cross-Ratios}

Our proof of mixing of the geodesic flow on Bowen-Margulis measures is inspired by Babillot's treatment for the smooth manifold case (\cite{babillot}), which involves the cross-ratio for endpoints of geodesics.

\begin{standing hypothesis}
In this section, let $\Gamma$ be a group acting by isometries on a proper CAT($0$) space $X$.
\end{standing hypothesis}

First we need to describe the space where cross-ratios will be defined.

\begin{definition}\label{quadrilaterals}
For $v^-, w^-, v^+, w^+ \in \bd X$, call $(v^-, w^-, v^+, w^+)$ a \defn{quadrilateral} if there exist rank one geodesics with endpoints $(v^-, v^+)$, $(w^-, w^+)$, $(v^-, w^+)$, and $(w^-, v^+)$.  Denote the set of quadrilaterals by $\QE$, and let $\QEL = \QE \cap \; \Lam^4$.
\end{definition}

\begin{definition}\label{cross-ratio def}
For a quadrilateral $(v^-, w^-, v^+, w^+)$, define its \defn{cross-ratio} by
\[\cratio (v^-, w^-, v^+, w^+) = \beta_p (v^-, v^+) + \beta_p (w^-, w^+) - \beta_p (v^-, w^+) - \beta_p (w^-, v^+),\]
for $p \in X$ arbitrary.
\end{definition}

Note that we removed reference to $p \in X$ in writing $\cratio$ in \defref{cross-ratio def}.  This omission is justified by the following lemma.

\begin{lemma}\label{no p in cratio}
The cross-ratio $\cratio (\xi, \xi', \eta, \eta')$ of a quadrilateral $(\xi, \xi', \eta, \eta')$ does not depend on choice of $p \in X$.
\end{lemma}
\begin{proof}
Let $v_0 \in \emap^{-1} (\xi, \eta)$, $v_1 \in \emap^{-1} (\xi, \eta')$, $v_2 \in \emap^{-1} (\xi', \eta')$, and $v_3 \in \emap^{-1} (\xi', \eta)$.  By \lemref{beta} and the definition of the cross-ratio,
\begin{align*}
\cratio (\xi, \xi', \eta, \eta')
&= [b_\xi (v_0 (0), p) + b_\eta (v_0 (0), p)] + [b_{\xi'} (v_2 (0), p) + b_{\eta'} (v_2 (0), p)] \\
&\qquad - [b_\xi (v_1 (0), p) + b_{\eta'} (v_1 (0), p)] - [b_{\xi'} (v_3 (0), p) + b_\eta (v_3 (0), p)]
\end{align*}
for any $p \in X$.  Using the cocycle property of Busemann functions, this gives us
\begin{align*}
\cratio (\xi, \xi', \eta, \eta')
&= b_\xi (v_0 (0), v_1 (0)) + b_\eta (v_0 (0), v_3 (0)) \\
&\qquad + b_{\eta'} (v_2 (0), v_1 (0)) + b_{\xi'} (v_2 (0), v_3 (0)),
\end{align*}
which is independent of $p \in X$.
\end{proof}

Our main interest in the cross-ratio is its connection with the horospherical foliations of $\SXL$.  The next lemma, due to Otal (\cite{otal92}), describes how the cross-ratio detects, to some extent, the non-integrability of the stable and unstable horospherical foliations.

\begin{lemma}\label{twisting}
Suppose $(\xi, \xi', \eta, \eta') \in \QE$.  Let $v_0 \in \emap^{-1} (\xi, \eta)$, and recursively choose $v_1 \in H^u (v_0)$ with $v_1^+ = \eta'$, $v_2 \in H^s (v_1)$ with $v_2^- = \xi'$, $v_3 \in H^u (v_2)$ with $v_3^+ = \eta$, and $v_4 \in H^s (v_3)$ with $v_4^- = \xi$.  Then $v_4 \sim g^{t_0} v_0$, for $t_0 = \cratio (\xi, \xi', \eta, \eta')$.
\end{lemma}
\begin{proof}
As in the proof of \lemref{no p in cratio}, we know
\begin{align*}
\cratio (\xi, \xi', \eta, \eta')
&= b_\xi (v_0 (0), v_1 (0)) + b_\eta (v_0 (0), v_3 (0)) \\
&\qquad + b_{\eta'} (v_2 (0), v_1 (0)) + b_{\xi'} (v_2 (0), v_3 (0)).
\end{align*}
But $b_\xi (v_0 (0), v_1 (0)) = b_{\eta'} (v_1 (0), v_2 (0)) = b_{\xi'} (v_2 (0), v_3 (0)) = b_\eta (v_3 (0), v_4 (0)) = 0$ by choice of $v_1, \dotsc, v_4$, so
\[\cratio (\xi, \xi', \eta, \eta')
= b_\eta (v_0 (0), v_3 (0))
= b_\eta (v_0 (0), v_4 (0))\]
by the cocycle property of Busemann functions.  On the other hand, $v_4 \parallel v_0$ by construction, and so by \propref{equivalent-geodesics}, $v_4 \sim g^{t} v_0$ for the value $t \in \R$ such that $b_\eta (v_0 (t), v_4 (0)) = 0$.  But
\[b_\eta (v_0 (t), v_4 (0)) = -t + b_\eta (v_0 (0), v_4 (0)) = -t + \cratio (\xi, \xi', \eta, \eta')\]
for all $t$, which shows $v_4 \sim g^{t_0} v_0$ for $t_0 = \cratio (\xi, \xi', \eta, \eta')$.
\end{proof}

The following proposition summarizes some of the basic properties of the cross-ratio (cf.~ \cite{ham97}).  The proofs are straightforward.

\begin{prop}
The cross-ratio on $\QE$ is continuous and satisfies all the following.
\begin{enumerate}
\item
$\cratio$ is invariant under the diagonal action of $\Isom X$ on $(\bd X)^4$,
\item
$\cratio (\xi, \xi', \eta, \eta') = -\cratio (\xi, \xi', \eta', \eta)$,
\item
$\cratio (\xi, \xi', \eta, \eta') = \cratio (\eta, \eta', \xi, \xi')$,
\item \label{cratio cocycle}
$\cratio (\xi, \xi', \eta, \eta') + \cratio (\xi, \xi', \eta', \eta'') = \cratio (\xi, \xi', \eta, \eta'')$, and
\item \label{triple identity for cratio}
$\cratio (\xi, \xi', \eta, \eta') + \cratio (\xi', \eta, \xi, \eta') + \cratio (\eta, \xi, \xi', \eta') = 0$.
\end{enumerate}
\end{prop}

We now relate cross-ratios to hyperbolic translation lengths (see \propref{cross-ratios <=> lengths}).
Write $\ell(\gamma)$ for the \defn{translation length} $\ell(\gamma) = \inf_{x \in X} d(x, \gamma x)$ of any $\gamma \in \Isom X$.  If there is some $x \in X$ such that $d(x, \gamma x) = \ell(\gamma) > 0$, we say $\gamma$ is \defn{hyperbolic}.  In this case, $x = v(0)$ for some geodesic $v \in SX$ with $\gamma v = g^{\ell(\gamma)} v$ (an \defn{axis} of $\gamma$).  For any hyperbolic isometry $\gamma \in \Isom X$, write $\gamma^+ = v^+$ and $\gamma^- = v^-$ for some (any) axis $v$ of $\gamma$.

\lemref{lengths from cross-ratios} shows how the translation length of any hyperbolic isometry of $X$ is given by some appropriately chosen cross-ratio, up to a factor of $2$.  For negatively curved manifolds, the result is due to Otal (\cite{otal92}).

\begin{lemma}\label{lengths from cross-ratios}
Let $\gamma$ be a hyperbolic isometry of $X$.  Then
\[\cratio (\gamma^-, \gamma^+, \gamma \xi, \xi) = 2 \ell(\gamma)\]
for all $\xi \in \bd X$ that are Tits distance $> \pi$ from both $\gamma^-$ and $\gamma^+$.
\end{lemma}
\begin{proof}
The proof outlined by Dal'bo (\cite{dalbo99}) for Fuchsian groups extends to CAT($0$) spaces.
\end{proof}

By \lemref{lengths from cross-ratios}, we can calculate the translation length of any hyperbolic isometry of $X$ in terms of cross-ratios.  The next lemma shows that we can calculate any cross-ratio in $\QE$ in terms of translation lengths of hyperbolic isometries of $X$.  The result is due to Kim (\cite{kim01}) and Otal (\cite{otal92}) for negatively curved manifolds.

\begin{lemma}\label{cross-ratios from lengths}
Let $\gamma_1, \gamma_2 \in \Gamma$ be rank one hyperbolic isometries with $\gamma_1^-$, $\gamma_1^+$, $\gamma_2^-$, and $\gamma_2^+$ all distinct.  Then
\[\cratio (\gamma_1^-, \gamma_2^-, \gamma_1^+, \gamma_2^+)
= \lim_{n \to \infty} \left[ \ell(\gamma_1^n) + \ell(\gamma_2^n) - \ell(\gamma_1^n \gamma_2^n) \right].\]
\end{lemma}
\begin{proof}
The proof given by Dal'bo (\cite{dalbo99}) for Fuchsian groups extends.
\end{proof}

Call $\setp{\ell(\gamma)}{\gamma \in \Gamma \text{ is hyperbolic}}$ the \defn{length spectrum}, and call $\cratio (\QEL)$ the \defn{cross-ratio spectrum}.  Say that either one is \defn{arithmetic} if lies in a discrete subgroup of $\R$.

\begin{prop}\label{cross-ratios <=> lengths}
Let $X$ be a proper \textnormal{CAT($0$)} space under a proper, non-elementary, isometric action by a group $\Gamma$ with a rank one element.  The cross-ratio spectrum is arithmetic if and only if the length spectrum is arithmetic.
\end{prop}
\begin{proof}
Let $L$ be the length spectrum.  Since the rank one axes are weakly dense in $\SXL$, Lemmas \ref{lengths from cross-ratios} and \ref{cross-ratios from lengths} give us
\[2L \subseteq \cratio (\QEL) \subseteq \cl{\langle L \rangle},\]
where $\cl{\langle L \rangle}$ is the closed subgroup in $\R$ generated by $L$.  This proves the proposition.
\end{proof}

We conclude this section with a variation on \lemref{lengths from cross-ratios} which is particularly useful for trees:  \lemref{lengths from cross-ratios} implies that $\cratio (\QE)/2$ contains all the translation lengths of hyperbolic elements of $\Isom X$.  If $X$ is a geodesically complete tree (with no vertices of valence $2$), the following lemma implies the slightly stronger statement that $\cratio (\QE)/2$ contains all the edge lengths of $X$.  We will use this fact in the proof of \lemref{discrete cross-ratios give trees}.

\begin{lemma}\label{tree lengths}
Suppose $X$ is geodesically complete and the link of $p, q \in X$ each has $\ge 3$ components.  Then there is some $(\xi, \xi', \eta, \eta') \in \QE$ such that $\cratio (\xi, \xi', \eta, \eta') = 2 d(p, q)$.
\end{lemma}
\begin{proof}
Let $r = d(p, q)$, and let $\rho_p \colon \bd X \to \link{p}{X}$ and $\rho_q \colon \bd X \to \link{q}{X}$ be radial projection onto the links of $p$ and $q$.  Find geodesics $v, w \in SX$ such that
\begin{enumerate}
\item $v(0) = w(r) = p$ and $v(r) = w(0) = q$,
\item $\rho_p (v^-), \rho_p (w^+), \rho_p (v^+)$ lie in distinct components of $\link{p}{X}$, and
\item $\rho_q (v^+), \rho_q (w^-), \rho_q (w^+)$ lie in distinct components of $\link{q}{X}$.
\end{enumerate}
One easily verifies that $(v^-, w^-, v^+, w^+) \in \QE$ and $\cratio (v^-, w^-, v^+, w^+) = 2r$.
\end{proof}

\section{Ergodicity and Mixing}

We now prove ergodicity and characterize mixing.

\begin{standing hypothesis}
In this section, let $X$ be a proper CAT($0$) space under a proper, non-elementary, isometric action by a group $\Gamma$ with a rank one element.  Assume that $m_{\Gamma}$ is finite.
\end{standing hypothesis}

For a group $G$ acting measurably on a space $Z$, a $G$-invariant measure $\nu$ on $Z$ is \defn{ergodic} under the action of $G$ if every $G$-invariant measurable set $A \subseteq Z$ has either $\nu(A) = 0$ or $\nu(Z \setminus A) = 0$.  If $G$ preserves only the measure class of $\nu$, and every $G$-invariant set has either zero or full $\nu$-measure, $\nu$ is called \defn{quasi-ergodic}.

\begin{theorem}[\thmref{main ergodicity}]
\label{ergodicity}
\stateergodicity{true}
\end{theorem}
\begin{proof}
The proof is standard.  The classical argument by Hopf (\cite{hopf71}) gives \itemrefstar{erg: ergodicity} from \corref{strong stable sets}.  Thus every $\Gamma$- and $g^t$-invariant measurable function $f \colon SX \to \R$  is constant $m$-almost everywhere, giving us \itemrefstar{erg: diagonal}.  Since $\mu$ and $\double{\mu_x}$ are in the same measure class (see \corref{RE has full measure}), the diagonal action of $\Gamma$ on $(\dbX, \double{\mu_x})$ is quasi-ergodic.  It follows that the $\Gamma$-action on $(\bd X, \mu_x)$ is also quasi-ergodic.  
Because convex combinations of conformal densities (of dimension $\delta_{\Gamma}$) are conformal densities (of dimension $\delta_{\Gamma}$), we can derive uniqueness of the Patterson-Sullivan measure from ergodicity of $m_{\Gamma}$.  Thus we have \itemrefstar{erg: p-s}.
And \itemrefstar{erg: div. type} follows from \lemref{shadow lemma} because $\mu_x$ gives positive measure to the set of radial limit points of $\Gamma$.
\end{proof}

For a locally compact group $G$ acting measurably on a space $Z$, a finite $G$-invariant measure $\nu$ on $Z$ is \defn{mixing} under the action of $G$ if, for every pair of measurable sets $A, B \subseteq Z$, and every sequence $g_n \to \infty$ in $G$, we have $\nu(A \cap g_n B) \to \frac{\nu(A) \, \nu(B)}{\nu(Z)}$.

\begin{lemma}\label{mixing}
Suppose $\ZL$ is nonempty.  Then either $m_\Gamma$ is mixing under the geodesic flow $g^t_{\Gamma}$ on $\modG{SX}$, or the cross-ratio spectrum is arithmetic.
\end{lemma}

\begin{proof}
The proof of \cite[Theorem 1]{babillot} extends to our situation.
\end{proof}

Finally, we characterize (see \lemref{discrete cross-ratios give trees}) the spaces with arithmetic cross-ratio spectrum.

\begin{lemma}\label{detecting geodesics}
Suppose $p \in X$ and $\xi, \eta \in \bd X$.  Then $\beta_p (\xi, \eta) = 0$ if and only if $\angle_p (\xi, \eta) = \pi$.
\end{lemma}
\begin{proof}
Both statements are equivalent to the existence of a geodesic in $X$ that joins $\xi$ and $\eta$ and passes through the point $p$.
\end{proof}

\begin{lemma}\label{quasi-tree}
Suppose $X$ is geodesically complete and all cross-ratios on $\QE$ take values in a fixed discrete subgroup of the reals.  Then
$\RL = \SXL$.
\end{lemma}
\begin{proof}
Suppose, by way of contradiction, that some $v \in \SXL$ has $d_T (v^-, v^+) = \pi$.  Find $\xi, \eta \in \Lam$ isolated in the Tits metric such that $\xi, \eta, v^-, v^+$ are distinct.  Since the rank one axes are strongly dense in $\SXL$, there is a sequence $v_k \to v$ such that $v_k^-$ and $v_k^+$ both are isolated in the Tits metric for every $k$.  We may assume $v_k^-, v_k^+ \in \Lam \setminus \set{\xi, \eta, v^-, v^+}$, hence $(v_k^-, \xi, v_k^+, \eta), (v_k^-, \xi, v^{\pm}, \eta) \in \QEL$.  Then
\[\cratio (v_k^-, \xi, v_k^+, \eta) = \cratio (v_k^-, \xi, v^+, \eta) = \cratio (v_k^-, \xi, v^-, \eta)\]
for all sufficiently large $k$ by discreteness and continuity of cross-ratios on $\QE$.  Thus
\begin{equation} \label{quasi-tree eqn}
\beta_p (v_k^-, v_k^+) - \beta_p (\xi, v_k^+) = \beta_p (v_k^-, v^-) - \beta_p (\xi, v^-),
\end{equation}
where $p \in X$ is arbitrary.

Recall from \lemref{beta is continuous} that $\beta_p \colon \dbX \to [-\infty, \infty)$ is upper semicontinuous on $\dbX$.  By \lemref{beta}, $\beta_p (v^-, v^-) = -\infty$, hence $\beta_p (v_k^-, v^-) \to -\infty$, but $\beta_p (\xi, v^-)$ is finite, so we have
\[
\lim_{k \to \infty} \left( \beta_p (v_k^-, v_k^+) - \beta_p (\xi, v_k^+) \right)
= \lim_{k \to \infty} \beta_p (v_k^-, v^-) - \beta_p (\xi, v^-)
= -\infty.
\]
But $\set{\beta_p (\xi, v_k^+)}$ is bounded because $\beta_p (\xi, v_k^+) \to \beta_p (\xi, v^+)$ by continuity of $\beta_p$ on $\RE$.  Therefore, $\beta_p (v_k^-, v_k^+) \to -\infty$ by \eqref{quasi-tree eqn} and upper semicontinuity.

On the other hand, $\beta_p \circ \emap$ is continuous on $SX$.  For if $v_k \to v$ in $SX$, then $(v_k^-, v_k^+) \to (v^-, v^+)$ in $\dbX$, so $(b_{v_k^-} + b_{v_k^+}) \to (b_{v^-} + b_{v^+})$ uniformly on compact subsets.  Also, $v_k(0) \to v(0)$ in $X$, so $(b_{v_k^-} + b_{v_k^+}) (v_k(0), p) \to (b_{v^-} + b_{v^+}) (v(0), p)$.  Thus $\beta_p (v_k^-, v_k^+)$ converges to $\beta_p (v^-, v^+)$.

Hence $\beta_p (v_k^-, v_k^+)$ must converge to $\beta_p (v^-, v^+)$, which is finite by \lemref{beta}; but this contradicts $\beta_p (v_k^-, v_k^+) \to -\infty$.  Therefore, $\RL = \SXL$.
\end{proof}

\begin{rem}
Note the requirement that the cross-ratios be discrete on all of $\QE$, not just $\QEL$, in \lemref{quasi-tree}.  It is not clear that $\cratio (v_k^-, \xi, v^+, \eta) = \cratio (v_k^-, \xi, v^-, \eta)$ holds without this assumption.
\end{rem}

\begin{lemma}\label{discrete cross-ratios give trees}
Suppose $X$ is geodesically complete, $\Lam = \bd X$, and $B(\QE) \subseteq c\Z$ for some $c > 0$.  Then
$X$ is isometric to a tree with all edge lengths in $2c\Z$.
\end{lemma}
\begin{proof}
Suppose all cross-ratios on $\QE$ lie in $c\Z \subset \R$, for some $a > 0$.  We will prove that the link $\link{p}{X}$ of $p$ is discrete at every point $p \in X$.  So fix $p \in X$, and let $\rho \colon \bd X \to \link{p}{X}$ be radial projection.

For $\eta \in \bd X$, let $A_p (\eta) = \setp{\xi \in \bd X}{\angle_p (\xi, \eta) = \pi}$.  Clearly every $\rho(A_p (\eta))$ is closed in $\link{p}{X}$.  We claim every $\rho(A_p (\eta))$ is also open.  For if $\rho(A_p (\eta_0))$ is not open for some $\eta_0 \in \bd X$, there is a point $\xi_0 \in A_p (\eta_0)$ and a sequence $(\xi_k)$ in $\bd X$ such that $\angle_p (\xi_0, \xi_k) \to 0$ but each $\angle_p (\xi_k, \eta_0) < \pi$.  For each $\xi_k$, choose $\eta_k \in A_p (\xi_k)$.  Passing to a subsequence, $(\xi_k, \eta_k) \to (\xi'_0, \eta'_0) \in \dbX$.  By continuity of $\angle_p$, we have $\angle_p (\xi'_0, \eta'_0) = \lim_{k \to \infty} \angle_p (\xi_k, \eta_k) = \pi$ and $\angle_p (\xi_0, \xi'_0) = 0$.  Hence
\[\angle_p (\xi_0, \eta'_0) = \angle_p (\xi'_0, \eta'_0) = \pi
= \angle_p (\xi_0, \eta_0) = \angle_p (\xi'_0, \eta_0),\]
with the left- and right-most equalities coming from the triangle inequality.  Thus
\[\cratio (\xi_0, \xi'_0, \eta_0, \eta'_0) = \beta_p (\xi_0, \eta_0) + \beta_p (\xi'_0, \eta'_0) -\beta_p (\xi_0, \eta'_0) - \beta_p (\xi'_0, \eta_0)\]
equals zero by \lemref{detecting geodesics} (note $(\xi_0, \xi'_0, \eta_0, \eta'_0) \in \QE$ because $\Reg = SX$ by \lemref{quasi-tree}).  By discreteness and continuity of cross-ratios, we have a neighborhood $U \times V$ of $(\xi'_0, \eta'_0)$ in $\dbX$ such that $\cratio (\xi_0, \xi, \eta_0, \eta) = 0$ for all $(\xi, \eta) \in U \times V$ (note $(\xi_0, \xi, \eta_0, \eta) \in \QE$ by openness of $\RE$ in $\dbX$, provided $U \times V$ is chosen to be sufficiently small).  Thus for large $k$, since $(\xi_k, \eta_k) \in U \times V$, we have $\cratio (\xi_0, \xi_k, \eta_0, \eta_k) = 0$.  But we know $0 = \beta_p (\xi_0, \eta_0) = \beta_p (\xi_k, \eta_k)$, hence
\[0 = \cratio (\xi_0, \xi_k, \eta_0, \eta_k) = -\beta_p (\xi_0, \eta_k) - \beta_p (\xi_k, \eta_0).\]
Both terms on the right being nonnegative, they must both equal zero.  Hence we have $\angle_p (\xi_k, \eta_0) = \pi$, contradicting our assumption on $\xi_k$.  Thus every $\rho(A_p (\eta))$ must be open in $\link{p}{X}$.

It follows from the previous paragraph that no component of $\link{p}{X}$ can contain points distance $\ge \pi$ apart.  But $\link{p}{X}$ is geodesically complete by \propref{geodesic completeness}, and no closed geodesic in $\link{p}{X}$ can have diameter less than $\pi$ because $\link{p}{X}$ is CAT($1$).  Thus $\link{p}{X}$ must be discrete for every $p \in X$.  Therefore $X$, being proper and geodesically complete, must be a metric simplicial tree.  So $2c\Z$ includes all edge lengths of $X$ by \lemref{tree lengths}.
\end{proof}

We now complete the proof of \thmref{main trees}.

\begin{theorem}[\thmref{main trees}]
\label{trees}
\statetrees{true}
\end{theorem}
\begin{proof}
\propref{cross-ratios <=> lengths} shows
\itemrefstar{condition: cross-ratios arithmetic}%
$\iff$%
\itemrefstar{condition: length spectrum arithmetic},
\lemref{mixing} shows
\itemrefstar{condition: not mixing}%
$\implies$%
\itemrefstar{condition: cross-ratios arithmetic}, and
\lemref{discrete cross-ratios give trees} shows
\itemrefstar{condition: cross-ratios arithmetic}%
$\implies$%
\itemrefstar{condition: special tree}.  If $X$ is a tree with all edge lengths in $c \Z$, then the geodesic flow factors continuously over the circle, so $m_\Gamma$ is not even weak mixing; this proves
\itemrefstar{condition: special tree}%
$\implies$%
\itemrefstar{condition: not mixing}.
\end{proof}

\begin{rems}
As observed in the introduction:
(a)  Mixing is also equivalent to weak mixing and to topological mixing under the hypotheses of \thmref{trees}.
(b)  The theorem does not hold if the geodesic completeness hypothesis is removed.
\end{rems}

\section{Quasi-Product Measures}
\label{gibbs measures}

Many of the results in previous sections hold not just for the Bowen-Margulis measure, but for a number of measures on $\modG{SX}$ called quasi-product measures.  For negatively curved manifolds, there are many examples of quasi-product measures, e.g.~ Gibbs measures, Liouville measures, harmonic measures, and Bowen-Margulis measures.  We state the more general results in this section.

\begin{standing hypothesis}
In this section, let $X$ be a proper CAT($0$) space under a proper, non-elementary, isometric action by a group $\Gamma$ with a rank one element.  In contrast to previous sections, however, the measures $m$, $m_{\Gamma}$, and $\mu$ are not necessarily associated to Bowen-Margulis or Patterson-Sullivan measures.
\end{standing hypothesis}

\begin{definition}\label{qp measure}
Call a $g^t$-invariant Borel measure $m_{\Gamma}$ on $\modG{\SXL}$ a \defn{quasi-product measure} if $m_{\Gamma}$ is supported on $\modG{\ZL}$ and the associated geodesic current $\mu = \emap_* (m)$---where $m$ is the Borel measure on $SX$ defined by \eqref{qm cond for B-M measure}---is equivalent to a finite product measure $\nu_1 \times \nu_2$ on $\dbX$.
\end{definition}

Since $\Gamma$ acts minimally on $\Lam$, we have $\supp(\nu_1) = \supp(\nu_2) = \Lam$ and therefore $\supp (\mu) = \dbL$.  Thus we obtain (cf. \propref{full support on SX}):

\begin{prop}
Suppose $\ZL$ is strongly dense in $\SXL$.  Then $\supp(m) = \SXL$ for any quasi-product measure $m_{\Gamma}$.
\end{prop}

\propref{strong density} holds, as stated, for any quasi-product measure $m_{\Gamma}$.  And much of \thmref{ergodicity}:

\begin{theorem}
Let $X$ be a proper \textnormal{CAT($0$)} space under a proper, non-elementary, isometric action by a group $\Gamma$ with a rank one element.  Suppose $m_{\Gamma}$ is a finite quasi-ergodic measure.
Then both the following hold:
\begin{enumerate}
\item
The Bowen-Margulis measure $m_\Gamma$ is ergodic under the geodesic flow on $\modG{SX}$.
\item
The diagonal action of $\Gamma$ on $(\GE, \mu)$ is ergodic.
\end{enumerate}
\end{theorem}

\lemref{mixing} holds for a finite quasi-product measure $m_{\Gamma}$, and thus we obtain (cf. \thmref{trees}):

\begin{theorem}
Let $X$ be a proper, geodesically complete \textnormal{CAT($0$)} space under a proper, non-elementary, isometric action by a group $\Gamma$ with a rank one element.  Suppose $m_{\Gamma}$ is a finite quasi-product measure and $\Lam = \bd X$.  The following are equivalent:
\begin{enumerate}
\item
The measure $m_\Gamma$ is not mixing under the geodesic flow on $\modG{SX}$.
\item
The length spectrum is arithmetic---that is, the set of all translation lengths of hyperbolic isometries in $\Gamma$ must lie in some discrete subgroup $c\Z$ of $\R$.
\item
There is some $c \in \R$ such that every cross-ratio of $\QE$ lies in $c\Z$.
\item
There is some $c > 0$ such that $X$ is isometric to a tree with all edge lengths in $c \Z$.
\item
The measure $m_\Gamma$ is not weak mixing under the geodesic flow on $\modG{SX}$.
\item
The geodesic flow on $\modG{SX}$ is not topologically mixing.
\end{enumerate}
\end{theorem}

Now, the condition in \defref{qp measure} that $m_{\Gamma}$ be supported on $\modG{\ZL}$ is important to be able to check.  Thus we consider a slightly weaker setting:

\begin{definition}\label{weak qp measure}
Call a $g^t$-invariant Borel measure $m_{\Gamma}$ on $\modG{(\GELR)}$ a \defn{weak quasi-product measure} if the associated geodesic current $\mu = \emap_* (m)$---where $m$ is the Borel measure on $\GER$ defined by \eqref{qm cond for B-M measure}---is equivalent to a finite product measure $\nu_1 \times \nu_2$ on $\dbX$.
\end{definition}

\lemref{recurrence} extends to this generality, giving us the following theorem (cf.~ \thmref{isolated almost everywhere}).

\begin{theorem}
Let $X$ be a proper \textnormal{CAT($0$)} space under a proper, non-elementary, isometric action by a group $\Gamma$ with a rank one element.  If $m_{\Gamma}$ is a finite weak quasi-product measure, then \mae{$\nu_1$} and \mae{$\nu_2$} $\xi \in \bd X$ is isolated in the Tits metric.
\end{theorem}

Lemmas \ref{constant a.e. 2} and \ref{unique-isometry-type} extend, and we obtain the statement of \thmref{a.e.-geodesic-is-lonely} verbatim.  Thus the distinction between finite quasi-product measures and finite weak quasi-product measures essentially vanishes when $\ZL$ is nonempty, as every finite weak quasi-product measure on $\modG{(\GELR)}$ is essentially a quasi-product measure on $\modG{\SXL}$.  More precisely, we have the following theorem:

\begin{theorem}
Let $X$ be a proper \textnormal{CAT($0$)} space under a proper, non-elementary, isometric action by a group $\Gamma$ with a rank one element, and suppose $\ZL$ is nonempty.  Then $(\pi_x)_*$ provides a canonical bijection from the set of finite quasi-product measures on $\modG{\SXL}$ to the set of finite weak quasi-product measures on $\modG{(\GELR)}$.
\end{theorem}

\bibliographystyle{amsplain}
\bibliography{mixing}

\end{document}